\documentclass{amsart}
\usepackage{amsmath,  amsfonts, latexsym, amscd, amssymb, graphicx}
\usepackage{cite}
\usepackage{tikz-cd}
\usepackage{tikz}

\graphicspath{ {./images/} }
\newtheorem{exm}{Example}[section]
\newtheorem{rmk}[exm]{Remark}
\newtheorem{prp}[exm]{Proposition}
\newtheorem{thm}[exm]{Theorem}
\newtheorem{cor}[exm]{Corollary}

\newtheorem{df}[exm]{Definition}

\newcommand{\n}{\noindent}
\newcommand {\cat}{\mathbf}

\def\Z{\mathbb{Z}}

\def\C{\mathbb{C}}

\begin{document}

%\markboth{Authors' Names}
%{Instructions for Typing Manuscripts (Paper's Title)}

%\catchline{}{}{}{}{}
%
%%%%%%%%%%%%%%%%%%%%%%%%%%%%%%%%%%%%%%%%%%%%%%%%%%%%%%%%%%%%%%%%%%%%

\title{Gyro-Groups, Gyro-splittings and Co-homology}

\author{Ramji Lal}

\address{Department of Mathematics, University of Allahabad, Allahabad, India\\
mathrjl@gmail.com}

\author{Vipul Kakkar}

\address{Department of Mathematics, Central University of Rajasthan, Kishangarh, India\\
vplkakkar@gmail.com}

\maketitle

%\begin{history}
%\received{(Day Month Year)}
%\revised{(Day Month Year)}
%\accepted{(Day Month Year)}
%\comby{(xxxxxxxxx)}
%\end{history}

\begin{abstract}
In this paper, we study gyro-groups associated to groups, group extensions admitting gyro-sections, and corresponding co-homologies. We also describe the  obstructions in terms of co-homomology. The notion of gyro-Schur Multiplier and that of gyro-Milnor $K_{2}$ group are introduced.
\end{abstract}


%\ccode{2000 Mathematics Subject Classification: 20J06, 20N05}

\section{Introduction}
Let $G$ be a group. We have an associated right loop $(G, \circ_1)$, where the binary operation $\circ_1$ is given by $x\circ_1y\ =\ y^{-1}xy^{2}$. The study of groups $G$ with prescribed properties on the associated  right loop $(G, \circ_1)$ was initiated by Foguel and Ungar \cite{tfau1,tfau2}. Indeed, they studied groups with prescribed properties on the associated left loop $(G, \circ)$ given by $x \circ y\ =\ x^{2}yx^{-1}$. However, for our convenience, we shall study it through the right loop structure $(G, \circ_1)$. It can be seen that $(G, \circ_1)$ is a right gyro-group \cite{laly1, laly2}.  Foguel and Ungar \cite{tfau2} showed that $(G, \circ_1)$ is a gyro-group if and only if $G$ is central by 2-Engel group. Gyro-groups have deep intrinsic relationship with twisted subroups, near subgroups \cite{mas}, and in turn, with the group theoretic subclass of constraint satisfaction problems \cite{tfmv}. The twisted version of right gyro-groups and subgroups has been studied in \cite{laly2}. A group $G$ is said to be weakly isomorphic or gyro-isomorphic to a group $K$ if $(G, \circ_1)$ is isomorphic $(K, \circ_1)$. A weak classification program was initiated in \cite{rlal}. More generally, a map $f$ from $G$ to $K$ will be termed as a gyro-homomorphism if $f(a\circ_1b)\ =\ f(a)\circ_1f(b)$ for all $a, b\in G$.  The main purpose of this paper is to introduce and  study the extensions  admitting  sections  which are gyro-homomorphisms. We also study the resulting co-homologies, obstructions, and  an analogue of Schur multiplier which will be  termed as Gyro-Schur multiplier. In turn, we introduce the notion of gyro-Milnor $K_{2}$-group. 

\section{Preliminaries}
This section is devoted to some basic notions, definitions and results.

 A magma $(S,\circ)$ with identity $e$ is called a right loop if the equation $X\circ a=b$ has a unique solution in $S$ for all $a,b\in S$.

Let $(S, \circ )$ be a right loop with identity $e$. For each $x,y,z\in S$, the unique solution to the equation
\begin{center}
$X\circ (y\circ z)\ =\ (x\circ y)\circ z$
\end{center}
will be denoted by $x\theta f(y, z)$. The map $f(y, z)$ from $S$ to $S$ defined by $f(y, z)(x)\ =\ x\theta f(y, z)$ is a member of the symmetric group $Sym (S)$ on $S$ which fixes $e\in S$.  Thus, $f(y, z)$ is a member of $Sym (S-\{e\}) \subset Sym (S)$ and which is termed as an inner mapping of $(S, \circ )$ determined by the pair $(y, z)\in S\times S$. Since we shall be dealing with right loops and right transversals, for convenience, we shall adopt the convention $(p\circ q)(x)\ =\ q(p(x))$ for the product in $Sym (S)$. The subgroup of $Sym (S)$ generated by the set $\{f(y, z)\mid y, z\in S\}$ of all inner mappings is termed as the inner mapping group (also termed as the group torsion) of the right loop $(S, \circ )$. We will denote the inner mapping group of the right loop $(S, \circ )$ by $G_S$. For each $y\in S$, let $R_{y}$ denote the right multiplication map on $S$ defined by $R_{y}(x)\ =\ x\circ y$. Clearly, $R_{y}\in Sym (S)$ for each $y\in S$ and the map $R$ from $S$ to $Sym (S)$ defined by $R(y)\ =\ R_{y}$ is an injective map. Let $R(S)$ denote the subgroup of $Sym (S)$ generated by the set $\{R_{y}\mid y\in S\}$ of all right multiplications. This is called the right multiplication group of $(S, \circ )$. Since 
\begin{center}
$(f(y, z ) o R_{y\circ z})(x)\ =\ f(y, z)(x)\circ (y\circ z)\ =\ (x\circ y)\circ (z)\ =\ (R_{y} o R_{z} )(x)$
\end{center}
for all $x, y, z\in S$, $R_{y}oR_{z}\ =\ f(y, z) o R_{y\circ z }$ for all $y, z\in S$. Again, 
\begin{center}
$(x\theta f(y', y)^{-1} \circ y')\circ y\ =\ x \circ (y' \circ y)\ =\ y$
\end{center}
for all $x, y\in S $, where $y'$ denotes the left inverse of $y$. This means that 
\begin{center}
$R_{y}^{-1}\ =\ f(y', y)^{-1} o R_{y'}$
\end{center}
for all $y\in S$. In turn, it follows that $G_{S}S$ is a subgroup of $R(S)$, where $S$ has been identified with the set $\{R_{y}\mid y\in S\}$ through the map $R$. Consequently, $R(S)\ =\ G_{S}S$. Since $G_{S}\bigcap S\ =\ \{I_{S}\}$, $S$ is a right transversal to $G_{S}$ in  $G_{S}S$. The group $G_{S}S$ is called the group extension ( also called the right multiplication group) of $S$.  Finally, $G_{S}S$ is universal in the sense that if $G$ is any group in which $(S, \circ )$ appears as a right transversal to a subgroup  of $G$, then there is a unique group homomorphism from $G_{S}S$ to $G$ which is identity on $S$ (see Theorem 3.4 \cite{lal1996transversals}).

\begin{df}(\cite{tfau1, laly1})
 A magma $(S, \circ)$ with a right identity $e$ is termed as a right gyro-group if the following four conditions  hold:
\begin{enumerate}
	\item[(i)] For each element $a\in S$, there is a right inverse $a'\in S$ with respect to $e$ in the sense that $a\circ a'\ =\ e$.
	\item[(ii)] For each $x, y, z\in S$, there is a unique element $x\theta f(y, z)\in S$ such that
     \[(x\circ y)\circ z\ =\ x\theta f(y, z) \circ (y\circ z).\]
	\item[(iii)] The map $f(y, z)$ from $S$ to $S$ given by $f(y, z)(x)\ =\ x\theta f(y, z)$ is an automorphism of $(S, \circ)$.
	\item[(iv)] For all $y\in S$, $f(y, y')\ =\ I_{S}$, where $I_S$ is the identity map on $S$.
\end{enumerate}
\end{df}

\n The following proposition gives us a necessary and sufficient condition for a magma to be a right gyro-group.

\begin{prp}(\cite{laly1})\label{s2p1}
 A magma $(S, \circ)$ is a right gyro-group if and only if $(S, \circ)$ is a right loop with identity such that all inner mappings $f(x, y)\in Aut (S, \circ)$ and $f(x', x)\ =\ I_{S}$, where $x'$ denotes the left inverse of $x$. $\sharp$
\end{prp}

\begin{df}(\cite{tfau1, laly1})
 A right transversal $S$ to a subgroup $H$ of the group $G$ containing the identity $e$ of $G$ is called a gyro-transversal if $S\ =\ S^{-1}\ =\ \{x^{-1}\mid x\in S\}$ and $h^{-1}xh\in S$ for all $x\in S$ and $h\in H$.
\end{df}

\n The following proposition relates right gyro-groups and gyro-transversals.

\begin{prp}(\cite{laly1})\label{s2p2}
 (Representation Theorem for Right Gyro-groups) A right loop $(S, \circ)$ is a right gyro-group if and only if it is a gyro-transversal to the right inner mapping group (group torsion) $G_{S}$ of $S$ in its group extension (right multiplication group) $G_{S}S$. $\sharp$
\end{prp}

\n For all the undefined terms of the cohomology theory in this paper, we refer \cite[Chapter 10]{lalb}.

\section{Gyro-groups and Gyro-transversals}

Consider a group $G$ and the semidirect product $\hat{G}\ =\  G\rtimes Inn (G)$ of $G$ with $Inn(G)$, where $Inn(G)$ denotes the group of inner automorphisms of $G$. An element of $\hat{G}$ is uniquely expressible in the form $(x, \alpha)$, where $x\in G$ and $\alpha\in Inn(G)$. The product $\cdot$ is given by $(x, \alpha )\cdot (y, \beta )\ =\ (x\alpha (y), \alpha\beta)$.  Every element $(x, \alpha )$ is uniquely expressible as $(x, \alpha )\ =\ (e, \alpha )(\alpha^{-1}(x), I_{G})$. Thus, $S\ =\ G\times \{I_{G}\}$ is a right transversal to $\{e\}\times Inn (G)$ in $\hat{G}$. The induced right loop structure on $S$ is the group structure on $S$. Since $S$ is a normal subgroup of $\hat{G}$, it is a gyro-transversal.  Further, an arbitrary right transversal  to $\{e\}\times Inn (H)$ in $\hat{G}$ is of the form $S_{g}\ =\ \{(e,g(x))\cdot (x, I_{G})\ =\ (g(x)(x), g(x))\mid x\in G\}$, where $g$ is a map from $G$ to $Inn (G)$ with $g(e)\ =\ I_{G}$.    Further, 
 
\begin{center}
$(g(x)(x), g(x))(g(y)(y), g(y))\ =\ (e, \alpha )(g(z)(z), g(z))$,
\end{center}
where $z\ =\ g(y)^{-1}(x)y$ and $\alpha\ =\ g(x)g(y)g(z)^{-1}$. Hence, the induced right loop operation $\circ_g$ on $S_{g}$ is given by
\begin{center}
$(g(x)(x), g(x))\circ_g(g(y)(y), g(y))\ =\ (g(z)(z), g(z))$,
\end{center}
where $z\ =\ g(y)^{-1}(x)y$. Clearly, the bijective map $x\mapsto (g(x)(x), g(x))$ from $G$ to $S_{g}$ induces a right loop structure  $\hat{\circ}_{g}$ on $G$ which is given by 
\begin{center}
$x\hat{\circ}_{g}y\ =\ g(y)^{-1}(x)y$.
\end{center}

\n Evidently, $(S_{g}, \circ_g)$ is isomorphic to $(G, \hat{\circ}_{g})$. It follows from \cite[Lemma 5.11]{laly1} that $S_{g}$ is a gyro-transversal if and only if $g(x^{-1})\ =\ g(x)^{-1}$ and $g$ is equivariant in the sense that   $g(\alpha^{-1}(x))\ =\ \alpha^{-1}g(x)\alpha$ for all $x\in G$ and $\alpha\in Inn(G)$. In turn, it also follows \cite[Proposition 5.10]{laly1} that $(S_g, \circ)$ and so also $(G, \hat{\circ}_{g})$ is a right gyro-group if and only if  $g(x^{-1})\ =\ g(x)^{-1}$ and $g$ is equivariant in the sense that  $g(\alpha^{-1}(x))\ =\ \alpha^{-1}g(x)\alpha$ for all $x\in G$ and $\alpha\in Inn(G)$. Now, every map $g$ from $G$ to $Inn (G)$ is determined by a map $\lambda$ from $G$ to $G$ with $\lambda (e)\ =\ e$ such that $g(x)\ =\ i_{\lambda (x)}$, where $i_{a}$ denotes the inner automorphism defined by $i_a(x)=axa^{-1}$.  To say that $(S_{g}, \circ_g)$ is a right gyro-group is to say that $i_{\lambda (x^{-1})}\ =\ i_{(\lambda (x)^{-1})}$ and $i_{\lambda (i_{b^{-1}}(x))}\ =\ i_{b^{-1}}i_{\lambda (x)}i_{b}$ for all $x, b\in G$. This, in turn, is equivalent to say that $\lambda (x^{-1})\lambda (x)$ and $\lambda (b^{-1}xb)b^{-1}\lambda (x)^{-1}b$ belong to the center $Z(G)$ for all $x, b\in G$. In particular, if a map $\lambda $ satisfies the conditions (i)   $\lambda (x^{-1})\ =\ \lambda (x)^{-1}$,  and (ii) $\lambda$ is equivariant in the sense that $\lambda (b^{-1}xb)\ =\ b^{-1}\lambda (x)b$ for all $x, b\in G$, then $S_{g}$ is a gyro-transversal and $(S_{g}, \circ_g)$ is a right gyro-group. In turn, $(G, \hat{\circ}_{g})$ is a right gyro-group, where $\hat{\circ}_{g}$ is given by

\begin{center}
$x\hat{\circ}_{g} y\ =\ i_{\lambda (y)^{-1}}(x)y\ =\ \lambda (y^{-1})x\lambda (y)y$,
\end{center}
$x, y\in G$. For each $n\in \mathbb{Z}$, the map $\lambda_{n}$ from $G$ to $G$ given by $\lambda (x)\ =\ x^{n}$ satisfies the above two conditions. Consequently, for each $n$, we get a right gyro-group structure $\circ_n$ on $G$ which is given by 
\begin{center}
$x\circ_n y\ =\ i_{y^{-n}}(x)y\ =\ y^{-n}xy^{n+1}$.
\end{center}
We shall be interested in right gyro-groups $(G, \circ_1)$. 
\begin{df} 
A right loop $(S, \circ)$ will termed as a group based right loop if it is isomorphic to a sub right loop of $(G, \circ_1)$ for some group $G$. 
\end{df}

\n The category of group based right loops will be denoted by $\cat{GR}$. Note that a group need not be a group based right loop. Indeed, a 3-group $G$ is a group based right loop if and only if  all elements of order 3 lie in the center of $G$ \cite[Corollary 5.4]{rlal}. Thus, a group of exponent 3 is group based right loop if and only if it is abelian. In particular,  the  non abelian group of order $3^3$ which is of exponent 3 is not a group based right loop. 

\begin{df}
A map $f$ from a group $G$ to a group $G'$ is said to be a gyro-homomorphism if $f$ is a homomorphism from $(G, \circ_1)$ to $(G', \circ_1)$. More explicitly, $f$ is said to be a gyro-homomorphism if $f(y^{-1}xy^{2})\ =\ f(y)^{-1}f(x)f(y)^{2}$ for all $x, y\in G$. A bijective gyro-homomorphism is called a gyro-isomorphism.
\end{df}

\n Evidently, a group homomorphism is a gyro-homomorphism. However, a gyro-homomorphism need not be a group homomorphism. For example, consider the extra special 3-group $G$ of exponent 3. Then $(G, \circ_1)$ is an  abelian group and the  identity map $I_{G}$ is a gyro-homomorphism from the group $G$ to the group $(G, \circ_1)$ which is not a group homomorphism. It also follows that gyro-isomorphic groups need not be isomorphic. We have a category $\cat{\hat{GP}}$ whose objects are groups and morphisms are gyro-homomorphisms. Evidently, the category $\cat{GP}$ of groups is a subcategory of $\cat{\hat{GP}}$ which is faithful but not full, and the category $\cat{\hat{GP}}$ is a faithful  subcategory of $\cat{GR}$ which is not full. The proof of the following proposition is straight forward.

\begin{prp}\label{s3p1}
Let $f$ be a gyro-homomorphism from a group $G$ to a group $G'$. Then the following hold:
\begin{enumerate}
	\item[($i$)] $f(e)\ =\ e$.
	\item[($ii$)] The power of an element considered as an element of $(G, \circ_1)$ is the same as that considered as an element of the group $G$.
	\item[($iii$)] $f(a^{n})\ =\ f(a)^{n}$ for all $a\in G$ and $n\in \Z$.
	\item[($iv$)] Image of a sub right loop of $(G, \circ_1)$ under $f$ is a sub right loop of $(G', \circ_1)$.
	\item[($v$)] Inverse image of a sub right loop (normal sub right loop) of $(G', \circ_1)$ under $f$  is a sub right loop (normal sub right loop) of $(G, \circ_1)$.
	\item[($vi$)] The fundamental theorem of gyro-homomorphisms hold in the category $\cat{\hat{GP}}$.
\end{enumerate}
\end{prp}

\n The proof of the fundamental theorem of gyro-group homomorphism can be found in \cite[Theorem 30, p. 418]{suksu}. Inverse image of a subgroup under $f$ need not be a subgroup. Consider the 3-exponent non-abelian group $G$ of order $3^{3}$. The identity map from $G$ to the elementary abelian  3-group  $(G, \circ_1)$ is a gyro-isomorphism. The number of subgroups of $(G, \circ_1)$ is 13 whereas the number of subgroups of $G$ is 4.

\begin{prp}\label{s3p2}
A map $f$ from $G$ to $G'$ is a gyro-homomorphism if and only if $f$ preserves identity and $f(y^{-1}xy^{2})\ =\ f(y^{-1})f(x)f(y^{2})$ for all $x, y\in G$.
\end{prp} 

\begin{proof}
Let $f$ be a gyro-homomorphism. From the previous proposition $f$ preserves identity and powers. Consequently,

 $f(y^{-1}xy^{2})\ =\ f(y)^{-1}f(x)f(y)^{2}\ =\ f(y^{-1})f(x)f(y^{2})$ for all $x, y\in G$. 

\vspace{0.2 cm}

\n Conversely, suppose that $f$ preserves the identity and $f(y^{-1}xy^{2}) = f(y^{-1})f(x)f(y^{2})$ for all $x, y\in G$. Putting $x = y$, we get that $f(y^{2})\ =\ f(y^{-1})f(y)f(y^{2})$. This shows that $f(y^{-1})\ =\ f(y)^{-1}$ for all $y\in G$. Further, putting $x\ =\ y^{-1}$, we get that $1\ =\ f(y^{-1})f(y^{-1})f(y^{2})$. This shows that $f(y^{2})\ =\ f(y)^{2}$ for all $y\in G$.
\end{proof}

\begin{prp}\label{s3p21}
An identity preserving map $t$ from $G$ to $G'$ is a gyro-homomorphism if and only if $\partial t(y^{-1}, x)\partial t(y^{-1}x, y^{2})\ =\ 1$, where the boundary map $\partial t$ is given by $\partial t(x, y)\ =\ t(x)t(y)t(xy)^{-1}$.
\end{prp}
\begin{proof}
Let $t$ be a map from $G$ to $G'$ which preserves identity. Then 
\begin{align*}
\partial t(y^{-1}, x)\partial t(y^{-1}x, y^{2}) = & t(y^{-1})t(x)t(y^{-1}x)^{-1}t(y^{-1}x)t(y^{2})t(y^{-1}xy^{2})^{-1} \\
 = & 1
\end{align*}
\n for all $x, y\in G$ if and only if $t(y^{-1}xy^{2})\ =\ t(y^{-1})t(x)t(y^{2})$ for all $x, y\in G$. The result follows from Proposition \ref{s3p2}.
\end{proof}
%%%%%%%%%%%%%%%%%%%%%%%%%%%%%%%%%%%%%%%%%%%%%%%%%%%%%%%%%%%%%%%%%%%%%%%%%%%%%%%%%%%%%

\section{Some Universal Constructions}

Let $X$ be a set and $F(X)$ be the  free group on $X$ consisting of the freely reduced words in $X$. Let $\hat{F}(X)$ denote the free group on $F(X)$ consisting of freely reduced words in $F(X)$. Usually, $\Omega$ will denote forgetful functors from a category to another category which forgets some structure.  

\begin{thm}\label{s4p1}
Let $\Omega$ denote the forgetful functor from the category $\cat{GR}$ of group based right loops to the category $\cat{RL}$ of right loops. Then there is a left adjoint to $\Omega$.
\end{thm}

\begin{proof} We construct the adjoint functor $\Sigma$ from $\cat{RL}$ to $\cat{GR}$. Let $(S, \circ)$ be a right loop. Consider the free group $F(S)$ on $S$ consisting of freely reduced words in $S$. Let $\hat{F(S)}$ denote the group having the presentation $\langle S ; R \rangle$ where $R\ =\ \{(x\circ y)^{-1}y^{-1}xy^{2}\}$. Let $\Sigma (S)$ denote the subset $\{y^{-1}xy^{2}\langle R \rangle\mid x, y\in S\}\ =\ \{(x\circ y)\langle R \rangle\mid x, y\in S\}$. Evidently $\Sigma (S)$ is a sub right loop of $(\hat{F(S)}, \circ_1)$, and hence it is a group based right loop. Clearly, the map $i_{S}$ from $S$ to $\Sigma (S)$ given by $i_{S}(x)\ =\ x \langle R \rangle$ is a homomorphism between right loops. 

\vspace{0.2 cm}

\n Let  $f$ be a homomorphism from $(S,\circ)$ to a group based right loop $(T, \circ_1)\subset (G, \circ_1)$.  From the universal property of a free group, we have a unique group homomorphism $\hat{f}$ from $F(S)$ to $G$ such that $\hat{f}(x)\ =\ f(x)$ for each $x\in S$. Since $f(x \circ y)\ =\ f(y)^{-1}f(x)f(y)^{2}$, $\hat{f}(x\circ y)\ =\ \hat{f}(y)^{-1}\hat{f}(x)\hat{f}(y)^{2}$ for all $x, y\in S$. This means that $R$ is contained in the kernel of $\hat{f}$. In turn, we have a unique group homomorphism $\overline{f} $ from $\hat{F(S)}$ to $G$. Evidently, $\overline{f}(\Sigma (S))\subseteq T$ and  $\overline{f}|_{\Sigma (S)}$ is the unique  homomorphism from $\Sigma (S)$ to $(T, \circ_1) $  such that $ \overline{f}|_{\Sigma (S)}\circ i_{S}\ =\ f$.

\vspace{0.2 cm} 

\n Next, let  $(S', \circ')$ be a right loop and $f$ be homomorphism from $(S, \circ)$ to $(S', \circ')$. Then $i_{S'}\circ f$ is a homomorphism from  $(S, \circ)$ to the group bases right loop $\Sigma (S')$, where $i_{S'}$ is the universal map  described in the above paragraph. Again from the universal property of $\Sigma (S)$ as described above, we have a unique homomorphism $\Sigma (f)$ from $\Sigma (S)$ to $\Sigma (S')$ such that $i_{S'}\circ f\ =\ \Sigma (f)\circ i_{S}$. Thus, $\Sigma$ defines a functor from the category $\cat{RL}$ to $\cat{GR}$. Finally, we need to show that the bi-functors $Mor (- , \Omega (-))$ and $Mor (\Sigma (-), -)$ from $\cat{RL\times GR}$ to the category $\cat{SET}$ of sets  are naturally isomorphic.  It follows from the above discussions that for each $(S, T)\in \cat{RL\times GR}$, we have the bijective map $\eta_{S, T}$ from $Mor (S, \Omega (T))$ to $Mor (\Sigma (S), T)$ given by $\eta_{S, T}(f)\ =\ \overline{f}|_{\Sigma (S)}$. The fact that $\eta\ =\ \{\eta_{S, T}\mid (S, T)\in Obj(RL)\times Obj (GR)\}$ is a natural isomorphism is an easy observation.
\end{proof}

\n Now, we construct free objects in the category $\cat{GR}$ of group based right loops. Let $X$ be a set. Consider the free group $F(X)$ on the set $X$ consisting of freely reduced words in $X$. If $W$ is a word in $X$, then $\overline{W}$ denotes the word in $X$ obtained by freely reducing $W$.  We define subsets $A_{n}, n\geq 0$ of $F(X)$ inductively as follows. Put $A_{0}$ to be the singleton $\{\overline{\emptyset}\ =\ 1\}$ consisting of the empty word representing the identity. Let $A_{1}\ =\ \{\overline{x^{\pm 1}}\mid x\in X\}$ be the set consisting of reduced words of length 1. Supposing that $A_{n}$ has already been defined, define $A_{n+1}\ =\ \{\overline{\overline{U}^{-1}\overline{V}\ \overline{U}^{2}}\mid \overline{U}, \overline{V}\in \bigcup_{i=0}^{n}A_{i}\}$. Evidently, $FR(X)\ =\ \bigcup_{i=1}^{\infty}A_{i}$ is a sub right loop of $(F(X), \circ_1)$ generated by $X$. The map $i$ from $X$ to $FR(X)$ given by $i(x)\ =\ \overline{x}$ is injective and the pair $(FR(X), i)$ is universal in the sense that if $j$ is a map from $X$ to a group based right loop $(T, \circ_1)\subseteq (G, \circ_1)$, then there is a unique homomorphism $\overline{j}$ from $FR(X)$ to $T$ such that $\overline{j}\circ i\ =\ j$. It follows that $FR$ defines a functor from the category $\cat{SET}$ of sets to the category $\cat{GR}$ which is adjoint to the forgetful functor $\Omega$. We shall term the $(FR(X) , i)$ as the free group based right loop on $X$. A pair $\langle X ; R \rangle$ together with a surjective homomorphism $f$ from $FR(X)$ to $(T, \circ_1)$ will be termed as a presentation of $T$ if the kernel of $f$ is the normal sub right loop of $FR(X)$ generated by $R$. Every group based right loop $(S, \circ_1)$ has the standard multiplication presentation induced by the obvious surjective homomorphism from $FR(S)$ to $S$. The cyclic group $\langle x \rangle$ considered as a group based right loop has a presentation $\langle\{x\} ; \emptyset \rangle$ and it is the universal free object in $\cat{GR}$. If $S$ and $T$ are group based right loops having presentations $\langle X; R \rangle$ and $\langle Y; S\rangle$, then the group based right loop  having the presentation $\langle X\bigcup Y; R\bigcup S \rangle$ is called the free product of $S$ and $T$, where $X\bigcup Y$ is taken as the disjoint union of $X$ and $Y$. Clearly, free objects in $\cat{GR}$ are free products of certain copies of universal free objects.

\vspace{0.2 cm}

\n Let $K$ be a group. Let $\langle K; R_{K} \rangle$ denote the standard multiplication presentation of $K$ and $\check{K}$ denotes  the group having the presentation $\langle K; \check{S}_{K} \rangle$, where $\check{S}_{K}$ is the set of words in $K$ of the type

\begin{center}
$(y^{-1}xy^{2})^{-1}\star y^{-1}\star x\star y^{2}$,
\end{center}

\n $x, y\in K -\{e\}$. Here the juxtaposition denotes the operation in the group $K$ and $\star$ denotes the operation in the free group $F(K)$ on $K$. More explicitly, $K\approx F(K)/\langle R_{K} \rangle$, where $\langle R_{K} \rangle$ is the normal subgroup of $F(K)$ generated by the set $R_{K}\ =\ \{(xy)^{-1}\star x\star y\mid x, y\in K\}$ and $\check{K}\approx F(K)/ \langle \check{S}_{K} \rangle$ where $\check{S}_{K}\ =\ \{(y^{-1}xy^{2})^{-1}\star y^{-1}\star x\star y^{2}\mid x, y\in K\}$. Clearly, $\langle \check{S}_{K} \rangle\subseteq \langle R_{K} \rangle$ and hence we have the surjective group homomorphism $\nu_K$ from $\check{K}$ to $K$ given by $\nu_K(x\langle \check{S}_K \rangle)=x \langle R_K \rangle$. The map $t_{K}$ from $K$ to $\check{K}$ given by $t_{K}(x)\ =\ x\langle\check{S}_{K} \rangle$ is an injective  gyro-homomorphism and $t_{K}(x^{n})\ =\ (t_{K}(x))^{n}$. If $f$ is a gyro-homomorphism from $K$ to a group $G$, then the map $\check{f}$ from $\check{K}$ to $G$ given by $\check{f}(x\langle \check{S}_K \rangle)=f(x)$ is the unique group homomorphism from $\check{K}$ to $G$ such that  $\check{f}\circ t_K=f$. Thus, the pair $(\check{K} , t_{K} )$ is universal in the sense that given any group $G$ and a gyro-homomorphism $f$ from $K$ to $G$, there is a unique group homomorphism $\check{f}$ from $\check{K}$ to $G$ such that $\check{f}\circ t_{K}\ =\ f$. Note that $f \circ \nu_{K}\circ t_{K}\ =\ f$ but $f\circ \nu_{K}$ need not be $\check{f}$ as it need not be a group  homomorphism (see Example \ref{exm1}). It also follows that the association $K\mapsto \check{K}$ defines a functor from the category $\cat{GP}$ to $\cat{\hat{GP}}$ which is adjoint to the forgetful functor, where $\cat{\hat{GP}}$ is a category whose objects are groups and the morphisms are gyro-homomorphisms. 

\vspace{0.2 cm}

\n Let $\check{R}_{K}\ =\  \langle R_{K} \rangle /\langle \check{S}_{K} \rangle$ and $\check{K}\ =\ F(K)/\langle \check{S}_{K} \rangle$. Then, we have the following short exact sequence

\begin{equation}\label{eqex}
1\longrightarrow \check{R}_{K}  \longrightarrow \check{K} \longrightarrow K\longrightarrow 1
\end{equation}

\n of groups having a section $t_{K}$ which is a gyro-homomorphism.

\vspace{0.2 cm}

\n More generally, let $\langle X; S \rangle$ be an arbitrary presentation of $K$. Consider the free group $F(F(X))$ on $F(X)$. We have a  surjective group homomorphism $\eta$ from $F(F(X))$ to $F(X)$ given by $\eta (W_{1}\star W_{2}\star\cdots \star W_{r})\ =\ W_{1}W_{2}\cdots W_{r}$, and $\langle F(X); \hat{S} \rangle$ is also a presentation of $K$, where $\hat{S}\ =\ \{W_{1}\star W_{2}\star\cdots \star W_{r}\mid W_{1}W_{2}\cdots W_{r}\in S\}$. Let $\check{T}$ denote the subset $\{(\eta (U^{-1}\star V\star U^{2})^{-1}\star U^{-1}\star V\star U^{2}\mid U, V\in F(X)\}$ of $F(F(X))$. It can be observed that $\langle \check{T} \rangle\subseteq \langle \hat{S} \rangle$. Consequently, we obtain a short exact sequence

\begin{equation}\label{eqex1}
1 \longrightarrow \langle \hat{S}  \rangle /\langle \check{T} \rangle\longrightarrow F(F(X))/\langle \check{T} \rangle\longrightarrow K\longrightarrow 1
\end{equation}

\n of groups which is equivalent to (\ref{eqex}). Indeed, if $\mu$ is the surjective homomorphism from $F(X)$ to $K$ given by the presentation $\langle X ; S \rangle$ of $K$, then it further induces a surjective group homomorphism $\check{\mu}$ from $F(F(X))$ to $F(K)$. It can be easily observed that $\check{\mu}(\langle\check{T} \rangle)\ =\ \langle \check{S}_K \rangle$. In turn, $\check{\mu}$ induces an isomorphism $\rho$ from $F(F(X))/\langle \check{T} \rangle$ to $\check{K}$ such that $(\rho^{-1} |_{\check{R}_{K}} , \rho^{-1} , I_{K} )$ is an equivalence from $(\ref{eqex})$ to $(\ref{eqex1})$.  In particular, $\check{K}\approx F(F(X))/\langle \check{T} \rangle$ and $\langle \hat{S}  \rangle /\langle \check{T} \rangle\approx \check{R}_{K}$. It follows that $F(F(X))/\langle \check{T} \rangle$ and $\langle \hat{S}  \rangle /\langle \check{T} \rangle$ are independent (up to isomorphism) of the presentation and they depend only on the group $K$. The associations $K\mapsto \check{K}$ and $K\mapsto \check{R}_{K}$ define functors from $\cat{GP}$ to itself which are universal in the sense already described.

The group $\check{R}_{K}$ can be thought of as the obstruction for gyro-homomorphisms from $K$ to be group homomorphisms. We also term it as a gyro-multiplier of $K$.
 
\begin{exm}
If $G$ is a cyclic group, then it is evident that $\check{G}\approx G$. Let $G$ be an elementary abelian 2-group. Then  $\check{G}$ has the presentation $ \langle G; \check{S}_{G} \rangle$, where $\check{S}_{G}\ =\ \{(y^{-1}xy^{2})^{-1}\star y^{-1}\star x\star y^{2}\mid x, y\in G-\{e\}\}\ =\ \{(yx)^{-1}\star (y\star x)\mid x, y\in G\}\ =\ R_{G}$. Thus, in this case also  $\check{G}\approx G$. Consider the quaternion group $Q_{8}\ =\ \{\pm 1, \pm i, \pm j, \pm k\}$. Evidently, $(j^{-1}ij^{2})^{-1}\star j^{-1}\star i\star j^{2}\ =\ (ji)^{-1}\star (j\star i)$ and so on. Indeed, $\check{S}_{Q_{8}}\ =\ R_{Q_{8}}$. Consequently, $\check{Q_{8}}\approx Q_{8}$ and $Q_{8}$ is gyro-isomorphic to itself.
\end{exm}

\begin{exm}\label{exm1}
Consider $G\ =\ \Z_{3}\times \Z_{3}\times \Z_{3}$. Since $G$ is of exponent 3, $\check{G}$ is also of exponent 3. Since $\check{G}$ is finitely generated, it is finite.  We show that $\check{G}$ is   non-abelian group. Let $E$ denote the non-abelian group of order $3^3$ which is of exponent 3. Since $E$ is nilpotent group of class 2 and of exponent 3, $(E, \circ_{1} )$ is an abelian group of exponent 3 and so it is isomorphic to  $ \Z_{3}\times \Z_{3}\times \Z_{3}$ as a group. In particular, we have a gyro-isomorphism $\eta$   from $G$ to $E$. From the universal property of $(\check{G}, t_{G})$, we get a surjective group homomorphism $\check{\eta}$ from $\check{G}$ to $E$ such that $\check{\eta}\circ t_{G}\ =\ \eta$. Since $E$ is non-abelian, $\check{G}$ is non-abelian. Again, since $G$ is abelian, $\check{R}_{G}$ contains the commutator $[ \check{G}, \check{G}]$ of $\check{G}$. Evidently, $\eta\circ \nu_{G}$ is not a group homomorphism as $(\eta\circ \nu_{G})^{-1}(\{1\})\ =\ \nu_{G}^{-1}(\{1\})\ =\ \check{R}_{G}\supseteq [ \check{G}, \check{G}]$ and $E$ is non-abelian. Note that $\eta\circ \nu_{G}\circ t_{G}\ =\ \eta$.   
\end{exm}

\begin{rmk}
From the Example \ref{exm1}, one observes that for the groups $G_1$ and $G_2$, $\check{(G_1\times G_2)}$ need not be isomorphic to $\check{G_1}\times \check{G_2}$. One can also observe that if $G_1$ is gyro-isomorphic to $G_2$, then $\check{G_1}$ is isomorphic to $\check{G_2}$ as groups. Even if  $\check{G_1}$ is isomorphic to $\check{G_2}$ as groups, then $G_1$ need not be gyro-isomorphic to $G_2$.
\end{rmk}

%\begin{exm}
%We describe $\check{S_{3}}$. Consider the multiplication presentation $\langle X; R \rangle$ of $S_{3}$, and the corresponding presentation $\langleX, \check{S} \rangle $ of $\check{S_{3}}$, where $X\ =\ \{x_{1}, x_{2}, x_{3}, x_{4}, x_{5}\}$ and $x_{1}\mapsto (1, 2), x_{2}\mapsto (1, 3), x_{3}\mapsto (2, 3), x_{4}\mapsto (1,2,3),\ and, x_{5}\mapsto (1,3,2)$. It can be verified that $R\ =\ \check{S}$. Consequently, the group $\check{S_{3}}$ is isomorphic to $S_{3}$. 
%\end{exm}

\begin{exm}
If $K$ is a free group on at least two generators, then it can be easily observed that the gyro-multiplier $\check{R}_{K}$ of $K$ is non-trivial, and  $t_{K}$ is gyro-homomorphism which is not a group homomorphism. 
\end{exm}

%%%%%%%%%%%%%%%%%%%%%%%%%%%%%%%%%%%%%%%%%%%%%%%%%%%%%%%%%%%%%%%%%%%%%%%%%%%%%%%%%%%%%%%%%%%%%%%%%%%%%%%%%%%

\section{Gyro-split extensions}

\begin{df}
A short exact sequence 
\begin{center}
$1\longrightarrow H\stackrel{\alpha}{\longrightarrow} G\stackrel{\beta}{\longrightarrow}K\longrightarrow 1$
\end{center}
of groups is called a gyro-split extension if there is a section $t$, also called a gyro-splitting, from $K$ to $G$ which is a gyro-homomorphism.
\end{df}

\n Evidently, a split extension is a gyro-split extension. However, a gyro-split extension need not be a split extension.

\begin{exm}
Consider the non-abelian group $E$ of order $3^{3}$ which is of exponent 3. Then $(E, \circ_1)$ is an elementary abelian 3-group and the identity map from $E$ to $(E, \circ_1)$ is a gyro-isomorphism. Consider the central extension
\begin{center}
$0\longrightarrow Z(E)\stackrel{i}{\longrightarrow} E\stackrel{\nu}{\longrightarrow}\Z_{3}\times \Z_{3}\longrightarrow 0$
\end{center}
of $\Z_{3}$ by $\Z_{3}\times \Z_{3}$. Evidently, it is not a split extension. However, there is a sub right loop $L$ of $(E, \circ_1)$ of order $3^2$ such that $E\ =\ Z(E)L$, and the map $\nu |_{L}$ is an isomorphism from $(L, \circ_1 )$ to $\Z_{3}\times \Z_{3}$. Indeed, there are $3^{2} + 3 + 1\ =\ 13$ subgroups  of $(E, \circ_1)\approx \Z_{3}^{3}$ of order $3^2$, whereas there are 4 subgroups of $E$ of order $3^2$. If $L$ is a subgroup $(E, \circ_1)$ of order $3^2$ which is not a subgroup of $E$, then  $L\bigcap Z(E)\ =\ \{1\}$. Consequently, $E\ =\ Z(E)L$ and the map $\nu |_{L}$ is an isomorphism from $(L, \circ_1)$ to $\Z_{3}\times \Z_{3}$. Evidently, $(\nu |_{L})^{-1}$ is a gyro-splitting.
\end{exm}

\begin{exm}
Let $K$ be an arbitrary field. Consider the unipotent group $U(3, K)$ of unipotent upper triangular $3\times 3$ matrices with entries in the field $K$. Then $U(3, K)$ is a nilpotent group of class 2. Thus, $(U(3, K), \circ_1)$ is a nilpotent group of class at most 2. Let $U(a_{1}, a_{2}, a_{3})$ denote the unipotent upper triangular $3\times 3$ matrix for which $a_{12}\ =\ a_{1},\ a_{13}\ =\ a_{2}$ and $a_{23}\ =\ a_{3}$. It can be easily observed that
 
\begin{center}
$U(b_{1}, b_{2}, b_{3})^{-1}U(a_{1}, a_{2}, a_{3})U(b_{1}, b_{2}, b_{3})^{2}\ =\ U(a_{1} + b_{1}, b_{2} + 2a_{1}b_{3} - b_{1}a_{3} + a_{2}, b_{3} + a_{3})$.
\end{center}
Thus, $(U(3, K), \circ_1)$ is isomorphic to the group $(K^{3}, \cdot )$, where the product $\cdot$ is given by
\begin{center}
$(a_{1}, a_{2}, a_{3})\cdot (b_{1}, b_{2}, b_{3})\ =\ (a_{1} + b_{1}, b_{2} + 2a_{1}b_{3} - b_{1}a_{3} + a_{2}, b_{3} + a_{3})$.
\end{center}
Evidently, $(U(3, K), \circ_1)$ is an algebraic group defined over the prime field of $K$. Further, $(U(3, K), \circ_1)$ is abelian if and only if the characteristic of $K$ is 3. Consider $U(3, \Z_{p} )$, where  $p$ is an odd prime different from 3. Then $U(3, \Z_{p})$ is a non abelian group of order $p^{3}$ and $(U(3, \Z_{p}), \circ_1)$ is also a non abelian group of order $p^{3}$ whose exponent is the same as that of $U(3, \Z_{p})$. It follows that $U(3, \Z_{p})$ is isomorphic to $(U(3, \Z_{p}), \circ_1)$. In other words $U(3, \Z_{p})$ is gyro-isomorphic to itself. Consequently, any gyro-split extension by $U(3, \Z_{p})$ is a split extension.  Further, note that  
\begin{center}
$0\longrightarrow Z(U(3, \Z_{p}))\stackrel{i}{\longrightarrow} U(3, \Z_{p})\stackrel{\nu}{\longrightarrow}\Z_{p}\times \Z_{p}\longrightarrow 0$
\end{center}
is not gyro-split.
\end{exm}

Using the universal property of the functor $G\mapsto \check{G}$, we can easily establish the following proposition:

\begin{prp}\label{s5p1}
To each short exact sequence of groups
\begin{center}
$E\equiv 1\longrightarrow H\stackrel{\alpha }{\longrightarrow}G\stackrel{\beta}{\longrightarrow}K\longrightarrow 1$,
\end{center}
 we have the following commutative diagram

\[
\begin{tikzcd}
              & 1 \arrow[d]                         & 1\arrow[d]                           & 1 \arrow[d]                     &  \\
  1 \arrow[r] & Ker\check{\beta}_{R} \arrow[d, "i"] \arrow[r, "i"] & \check{R_G} \arrow[d, "i_G"] \arrow[r, "\check{\beta_R}"] & \check{R_K} \arrow[d, "i_K"]  \\
  1 \arrow[r] & Ker \check{\beta} \arrow[d, "\nu"] \arrow[r, "i"]  & \check{G} \arrow[d, "\nu_G"] \arrow[r, "\check{\beta}"]   & \check{K} \arrow[d, "\nu_K"] \arrow[r] & 1 \\
  1 \arrow[r] & H  \arrow[r, "\alpha"]                             & G \arrow[d] \arrow[r, "\beta"]                            & K \arrow[d] \arrow[r] & 1 \\
	                                                                 &                                                          & 1 & 1 &                  
\end{tikzcd}
\]
 
where the rows and the columns are exact. Further, if the bottom row is gyro-split, then the middle row is split exact sequence.

\end{prp}

\begin{proof} Consider the right most gyro-split  vertical exact sequence. We have the gyro-splitting $t_{K}$ from $K$ to $\check{K}$, and $t_{K}\circ \beta$ is a gyro-homomorphism from $G$ to $\check{K}$. From the universal property of the pair $(\check{G},t_G)$, we have a unique group homomorphism  $\check{\beta}$ from $\check{G}$ to $\check{K}$ such that $\check{\beta}\circ t_G=t_K\circ \beta$. In turn,
\begin{center} 
$\nu_{K}\circ \check{\beta}\circ t_{G}\ =\ \nu_{K}\circ t_{K}\circ \beta\ =\ \beta\ =\ \beta\circ \nu_{G}\circ t_{G}$.
\end{center}
Since $\nu_{K}\circ \check{\beta}$ and $\beta\circ \nu_{G}$ are group homomorphisms from $\check{G}$ to $K$ and $\beta$ is a gyro-homomorphism (being a group homomorphism), it follows from the universal property of $(\check{G} , t_{G})$ that $\nu_{K}\circ \check{\beta}\ =\ \beta\circ \nu_{G}$. Thus the lower right square is commutative. Further, since $t_{K}(K)$ generates $\check{K}$ as a group and $\beta$ is surjective, it follows that $\check{\beta}$ is surjective.   Evidently, the diagram is commutative, all the rows and the last two columns are exact. The exactness of the first column also follows by chasing the diagram. Note that $\nu$ and $\check{\beta}_{R}$ need not be surjective. 

Finally, suppose that the bottom row is gyro-split with $t$ as gyro-splitting. Then $t_{G}\circ t$ is a gyro-homomorphism from $K$ to $\check{G}$. From the universal property of $(\check{K} , t_{K} )$, we have a unique group homomorphism $\check{t}$ from $\check{K}$ to $\check{G}$ such that $\check{t}\circ t_{K}\ =\ t_{G}\circ t$. In turn,
\begin{center}
$\check{\beta}\circ \check{t}\circ t_{K}\ =\ \check{\beta}\circ t_{G}\circ t\ =\ t_{K}\circ \beta \circ t\ =\ t_{K}\ =\ I_{\check{K}}\circ t_{K}$.
\end{center}
It follows from the universal property of $(\check{K} , t_{K} )$ that $\check{\beta}\circ \check{t}\ =\ I_{\check{K}}$.
\end{proof}

\begin{rmk}
Since $t_{G}|_{Ker \beta}$ is a gyro-homomorphism from $Ker \beta\ =\ im( \alpha )$ to $\langle t_{G}(Ker \beta ) \rangle\subseteq Ker\check{\beta}$, we have a unique group homomorphism $\check{\alpha}$ from $\check{H}$ to $\check{G}$ such that $ \check{\alpha}\circ t_{Ker\beta}\ =\ t_{G}|_{Ker \beta}$. Evidently, $im( \check{\alpha})\subseteq Ker\check{\beta}$. However, the equality need not hold. In turn, we get a natural invariant $inv (E)\ =\ Ker\check{\beta}/im( \check{\alpha})$ associated to the extension $E$.
\end{rmk}

\n Let $\cat{GEXT}$ denote the category whose objects are gyro-split extensions and a morphism from a gyro-split extension
\begin{center}
$E\equiv 1\longrightarrow H\stackrel{\alpha }{\longrightarrow}G\stackrel{\beta}{\longrightarrow}K\longrightarrow 1$
\end{center}
to a gyro-split extension 
\begin{center}
$E'\equiv 1\longrightarrow H'\stackrel{\alpha '}{\longrightarrow}G'\stackrel{\beta '}{\longrightarrow}K'\longrightarrow 1$
\end{center}
is a triple $(\lambda , \mu , \nu )$, where $\lambda$ is a group homomorphism from $H$ to $H'$, $\mu$ is a  group homomorphism from $G$ to $G'$ and $\nu$ is a gyro-homomorphism from $K$ to $K'$ such that  the corresponding diagram is commutative. The composition of morphisms is obvious. Observe that in this context the short five lemma also holds. Thus, $(\lambda , \mu , \nu )$ is an equivalence if and only if $\lambda$ and $\nu$ are bijective. 

\begin{thm}\label{s5p2}
The gyro-split extension described in (\ref{eqex}), section 4 is a free gyro-split extension by  $K$ in the sense that if 
\begin{center}
$E\equiv 1\longrightarrow H\stackrel{\alpha }{\longrightarrow}L\stackrel{\beta}{\longrightarrow}K' \longrightarrow 1$
\end{center}
is a gyro-split extension by $K'$ and $\eta$ a group homomorphism from $K$ to $K'$, then there is a unique pair $(\lambda,\mu)$ of group homomorphisms such that the triple $(\lambda , \mu , \eta )$ is a morphism from the extension (\ref{eqex}) to $E$.
\end{thm}

\begin{proof} Let $s$ be a gyro-splitting of $E$. Then $s\circ \eta$ is a gyro-homomorphism from $K$ to $L$. From the universal property of $(\check{K} , t_{K} )$ we get a unique group homomorphism $\mu$ from $\check{K}$ to $L$ such that $\mu\circ t_{K}\ =\ s \circ \eta $. Hence
\begin{center}
$\beta \circ \mu \circ t_{K}\ =\ \beta\circ s \circ \eta\ =\ \eta\ =\ \eta \circ \nu_{K} \circ t_{K}$.
\end{center}
Since $\eta \circ \nu_{K}$ is a group homomorphism, it follows from the universal property of $(\check{K} , t_{K} )$ that $\beta\circ \mu\ =\ \eta \circ \nu_{K}$.
 Also $\beta \circ \mu \circ i\ = \eta \circ \nu_{K} \circ i\ =\ 0$, where $i$ is the inclusion from $\check{R}_{K}$ to $\check{K}$. Consequently, there is a unique group homomorphism $\lambda$ from $ \check{R}_{K} $ to $H$ such that $(\lambda , \mu , \eta )$ is a morphism in $\cat{GEXT}$. 
\end{proof}

\n Let
\begin{center}
$E\equiv 1\longrightarrow H\stackrel{\alpha }{\longrightarrow}G\stackrel{\beta}{\longrightarrow}K\longrightarrow 1$
\end{center}
be a gyro-split extension and $t$ be a gyro-splitting of $E$. We have the corresponding factor system $(K, H, \sigma^{t},f^{t})$, where $f^{t}$ is the map from $K\times K$ to $H$ given by $t(x)t(y)\ =\ \alpha (f^{t}(x, y))t(xy)$ and $\sigma^{t}$ is the map from $K$ to $Aut (H)$ given by $\alpha (\sigma^{t}(x)(h))\ =\ t(x)\alpha (h)t(x)^{-1}$. We denote $\sigma^{t}(x)$ by $\sigma^{t}_{x}$. Further, since $t$ is a gyro-homomorphism, $\sigma^{t}$ is a gyro-homomorphism (note that it need not be a group homomorphism) and

\begin{equation}\label{eq1}
f^{t}(y^{-1},x)f^{t}(y^{-1}x, y^{2})\ =\ 1\ =\ \sigma^t_{y^{-1}}(f^t(x, y^{2}))f^t(y^{-1}, xy^{2})
\end{equation}
 
\n for all $x, y\in K$. In particular $f^t(y, y^{-1})\ =\ 1$ for all $y\in K$. This prompts us to have the following definition:
\begin{df}
A factor system $(K, H, \sigma, f)$ will be called a gyro-factor system if $\sigma$ is a gyro-homomorphism from $K$ to $Aut (H)$ and $f$ satisfies (\ref{eq1}) with $f^{t}$ replaced by $f$. Such a map $f$ is also called a gyro-pairing. 
\end{df}

\n Let $(\lambda , \mu , \nu )$ be a morphism from a gyro-split extension

\begin{center}
$E\equiv 1\longrightarrow H\stackrel{\alpha }{\longrightarrow}G\stackrel{\beta}{\longrightarrow}K \longrightarrow 1$
\end{center}

\n to a gyro-split extension

\begin{center}
$E'\equiv 1\longrightarrow H'\stackrel{\alpha ' }{\longrightarrow}G'\stackrel{\beta '}{\longrightarrow}K' \longrightarrow 1$.
\end{center}

 Let $t$ be a gyro-splitting of $E$ and $t'$ be a gyro-splitting of $E'$. Since $\beta '(\mu ( t(x)))\ =\ \nu (\beta (t(x)))\ =\ \nu (x)\ =\ \beta '(t'(\nu (x)))$ for $x\in K$,  there is a unique map $g$ from $K$ to $H'$ with $g(1)\ =\ 1$ such that 

\begin{equation}\label{eq2}
\mu (t(x))\ =\ \alpha '(g(x))t'(\nu (x))
\end{equation}

\n for all $x\in K$. Since $t$ is a gyro-homomorphism,

\begin{equation}\label{eq3}
\mu (t(y^{-1})t(x)t(y^{2}))\ =\ \mu (t(y^{-1} x y^{2}))\ =\ \alpha '(g(y^{-1}xy^{2}))t'(\nu (y^{-1}xy^{2}))
\end{equation}
for all $x, y\in K$. Now, \\

\begin{align*}
\mu (t(y^{-1})t(x)t(y^{2})) = & \mu t (y^{-1})\mu t (x)\mu t(y^{2}) \\
                            = & \alpha '(g(y^{-1}))t'(\nu (y^{-1}))\alpha '(g(x))t'(\nu (x))\alpha '(g(y^{2}))t'(\nu (y^{2})) \text{ by }(5.2) \\
 = & \alpha '(g(y^{-1}))\alpha '(\sigma^{t'}_{\nu (y^{-1})}(g(x)))t'(\nu (y^{-1}))t'(\nu (x))\alpha '(g(y^{2}))t'(\nu (y^{2})) \\
= & \alpha '(g(y^{-1})\sigma^{t'}_{\nu (y^{-1})}(g(x))f^{t'}(\nu (y^{-1}), \nu (x)))t'(\nu (y^{-1})\nu (x))\alpha '(g(y^{2}))t'(\nu (y^{2})) \\
= & \alpha '(g(y^{-1})\sigma^{t'}_{\nu (y^{-1})}(g(x))f^{t'}(\nu (y^{-1}), \nu (x))\sigma^{t'}_{\nu (y^{-1})\nu (x)}(g(y^{2}))) \\
 & \quad \quad \quad \quad t'(\nu (y^{-1})\nu (x))t'(\nu (y^{2})) \\
= & \alpha '(g(y^{-1})\sigma^{t'}_{\nu (y^{-1})}(g(x))f^{t'}(\nu (y^{-1}), \nu (x))\sigma^{t'}_{\nu (y^{-1})\nu (x)}(g(y^{2})) \\
 & \quad \quad \quad \quad f^{t'}(\nu (y^{-1})\nu (x), \nu (y^{2})))t'(\nu (y^{-1}xy^{2}))\\
= & \alpha '(g(y^{-1})\sigma^{t'}_{\nu (y^{-1})}(g(x))\sigma^{t'}_{\nu (y^{-1})}(\sigma^{t'}_{\nu (x)}(g(y^{2})))f^{t'}(\nu (y^{-1}), \nu (x)) \\
 & \quad \quad \quad \quad f^{t'}(\nu (y^{-1})\nu (x), \nu (y^{2})))t'(\nu (y^{-1}xy^{2})) \\
= & \alpha '(g(y^{-1})\sigma^{t'}_{\nu (y^{-1})}(g(x))\sigma^{t'}_{\nu (y^{-1})}(\sigma^{t'}_{\nu (x)}(g(y^{2})))t'(\nu (y^{-1}xy^{2})) \text{ (by (\ref{eq1}))} 
\end{align*}

\n for all $x,y\in K$. Thus, comparing the both sides of Equation (\ref{eq3}), we obtain
\begin{equation}\label{eq4}
g(y^{-1}xy^{2})\ =\ g(y^{-1})\sigma^{t'}_{\nu (y^{-1})}(g(x)\sigma^{t'}_{\nu (x)}(g(y^{2})))
\end{equation}

\n for all $x,y\in K$. Further, 

\begin{align*}
\alpha '(\lambda (\sigma^{t}_{x}(h))) = & \mu (\alpha (\sigma^{t}_{x}(h)))\\
 = & \mu (t(x)\alpha (h) t(x)^{-1}) \\
 = & \mu (t(x))\alpha '(\lambda (h))\mu (t(x^{-1})) \\
 = & \alpha '(g(x))t'(\nu (x))\alpha '(\lambda (h))\alpha '(g(x^{-1}))t'(\nu (x^{-1})) \\
 = & \alpha '(g(x)\sigma^{t'}_{\nu (x)}(\lambda (h)g(x^{-1}))) \text{ since }t' \text{ and }\nu \text{ are gyro-homomorphisms.}
\end{align*}

\n Thus, 
\begin{equation}\label{eq5}
\lambda (\sigma^{t}_{x}(h))\ =\ g(x)\sigma^{t'}_{\nu (x)}(\lambda (h)g(x^{-1}))
\end{equation}
for all $x\in K$ and $h\in H$.

\n Let $(\lambda_{1}, \mu_{1}, \nu_{1})$ be a morphism from a gyro-split extension 

\begin{center}
$E_{1}\equiv 1\longrightarrow H_{1}\stackrel{\alpha_{1}}{\longrightarrow}G_{1}\stackrel{\beta_{1}}{\longrightarrow}K_{1}\longrightarrow 1$
\end{center}
to
\begin{center}
$E_{2}\equiv 1\longrightarrow H_{2}\stackrel{\alpha_{2}}{\longrightarrow}G_{2}\stackrel{\beta_{2}}{\longrightarrow}K_{2}\longrightarrow 1$,
\end{center}
and $(\lambda_{2}, \mu_{2}, \nu_{2})$ be a morphism from $E_{2}$ to a gyro-split extension
\begin{center}
$E_{3}\equiv 1\longrightarrow H_{3}\stackrel{\alpha_{3}}{\longrightarrow}G_{3}\stackrel{\beta_{3}}{\longrightarrow}K_{3}\longrightarrow 1$.
\end{center}
Let $t_{1}, t_{2} $ and $t_{3}$ be the corresponding choice of gyro-splittings. Then
\begin{center}
$\mu_{1}(t_{1}(x))\ =\ \alpha_{2}(g_{1}(x))t_{2}(\nu_{1}(x))$ 
\end{center}
\n for all $x\in K_1$ and
\begin{center}
$\mu_{2}(t_{2}(x))\ =\ \alpha_{3}(g_{2}(x))t_{3}(\nu_{2}(x))$
\end{center}
for all $x\in K_{2}$, where $g_{1}$ is the uniquely determined map from $K_{1}$ to $H_{2}$ and $g_{2}$ is the uniquely determined map from $K_{2}$ to $H_{3}$. In turn,
\begin{center}
$\mu_{2}(\mu_{1}(t_{1}(x)))\ =\ \alpha_{3}(g_{3}(x))t_{3}(\nu_{2} (\nu_{1}(x)))$,
\end{center}

\n where $g_{3}(x)\ =\ \lambda_{2}(g_{1}(x))g_{2}(\nu_{1}(x))$ for each $x\in K_{1}$. This introduces a category $\cat{GFAC}$ of gyro-factor systems whose objects are gyro-factor systems and a morphism from a gyro-factor system $(K_{1}, H_{1}, \sigma^{1}, f^{1})$ to $(K_{2}, H_{2}, \sigma^{2}, f^{2})$ is a triple $(\nu , g, \lambda )$,
where $\nu$ is a gyro-homomorphism from $K_{1}$ to $K_{2}$, $\lambda$ a group homomorphism from $H_{1}$ to $H_{2}$, and $g$ is a map from $K_{1}$ to $H_{2}$ such that

\begin{enumerate}
	\item[(i)] $g(1)\ =\ 1$,
	\item[(ii)] $g(y^{-1}xy^{2})\ =\ g(y^{-1})\sigma^{2}_{\nu (y^{-1})}(g(x)\sigma^{2}_{\nu (x)}(g(y^{2})))$ and
	\item[(iii)] $\lambda (\sigma^{1}_{x}(h))\ =\ g(x)\sigma^{2}_{\nu (x)}(\lambda (h)g(x^{-1}))$, 
\end{enumerate}

\n for all $x,y\in K_{1}$ and $h\in H_{1}$. The composition of a morphism $(\nu_{1}, g_{1}, \lambda_{1})$ with $(\nu_{2}, g_{2}, \lambda_{2})$ is $(\nu_{2}\circ \nu_{1}, g_{3}, \lambda_{2}\circ \lambda_{1})$, where $g_{3}(x)\ =\ \lambda_{2}(g_{1}(x))g_{2}(\nu_{1}(x))$ for all $x\in K_1$.

\vspace{0.2 cm}

\n Using the axiom of choice, we have a choice $t_{E}$ of a gyro-splitting of a gyro-split extension $E$. Evidently, the association $GFAC$ which associates to each gyro-extension $E$ the  gyro-factor system $GFAC(E, t_{E})$ associated to the section $t_{E}$ gives an equivalence between $\cat{GEXT}$ and $\cat{GFAC}$.

\vspace{0.2 cm}

\n Let us fix a pair $H$ and $K$ of groups. We try to describe the equivalence classes of gyro-split extensions of $H$ by $K$. Let $G$ be a gyro-split extension of $H$ by $K$ given by the exact sequence
\begin{center}
$E\equiv 1\longrightarrow H\stackrel{\alpha}{\longrightarrow}G\stackrel{\beta}{\longrightarrow}K\longrightarrow 1$.
\end{center}
Let $(\lambda , \mu , \nu )$ be an equivalence from $E$ to a gyro-split extension $G'$ of $H'$ by $K'$ which is given by the exact sequence
\begin{center}
$E'\equiv 1\longrightarrow H'\stackrel{\alpha '}{\longrightarrow}G'\stackrel{\beta '}{\longrightarrow}K'\longrightarrow 1$.
\end{center}
Then it is clear that $G'$ is also a gyro-split extension of $H$ by $K$ given by the exact sequence
\begin{center}
$E''\equiv 1\longrightarrow H\stackrel{\alpha '\circ \lambda}{\longrightarrow}G'\stackrel{\beta\circ \mu^{-1}}{\longrightarrow}K\longrightarrow 1$.
\end{center}
such that $E$ is equivalent to $E''$ and $E''$ is equivalent to $E'$. As such there is no loss of generality in restricting the concept of equivalence on the class $GE(H, K)$ of all gyro-split extensions of $H$ by $K$ by saying that 
\begin{center}
$E_{1}\equiv 1\longrightarrow H\stackrel{\alpha_{1}}{\longrightarrow}G_{1}\stackrel{\beta_{1}}{\longrightarrow}K\longrightarrow 1$.
\end{center}
and 
\begin{center}
$E_{2}\equiv 1\longrightarrow H\stackrel{\alpha_{2}}{\longrightarrow}G_{2}\stackrel{\beta_{2}}{\longrightarrow}K\longrightarrow 1$.
\end{center}
in $GE(H, K)$ are equivalent if there is an isomorphism $\phi$ from $G_{1}$ to $G_{2}$ such that $(I_{H}, \phi , I_{K})$ makes the corresponding diagram commutative.

\begin{prp}\label{s5p3}
An abstract kernel $\psi$ from $K$ to $Out (H)$ is realizable from a gyro-split extension if and only if the obstruction $Obs(\psi )\in H^{3}_{\sigma}(K, Z(H))$ is 0 and $\psi$ has a lifting from $K$ to $Aut (H)$ which  is a gyro-homomorphism. Here $\sigma$ is a group homomorphism from $K$ to $Aut (Z(H))$ induced by $\psi$.
\end{prp}

\begin{proof} We already know that $\psi$ is realizable from an extension if and only if $Obs (\psi )\ =\ 0$ (see \cite[Proposition 10.2.1, p. 392]{lalb}). Further, then,  it is realizable from a gyro-split extension 

\begin{center}
$1\longrightarrow H\stackrel{\alpha}{\longrightarrow}G\stackrel{\beta}{\longrightarrow}K\longrightarrow 1$
\end{center}
 if and only if there is a gyro-splitting $t$ such that $\psi (x)\ =\ \sigma^{t}_{x}Inn(H)$ for each $x\in K$. Since $t$ is a gyro-splitting, $\sigma^{t}$ is a lifting of $\psi$ which is a gyro-homomorphism.
\end{proof}

\n The following two corollaries are immediate.

\begin{cor}\label{s5p3c1}
An abstract kernel $\psi$ from $K$ to $Out (H)$ is realizable from a gyro-split extension if and only if the obstruction $Obs(\psi )\in H^{3}_{\sigma}(K, Z(H))$ is 0 and the short exact sequence
\begin{center}
$0\longrightarrow Inn (H)\stackrel{i_{1}}{\longrightarrow}Aut (H)\times_{(\nu , \psi )}K\stackrel{p_{2}}{\longrightarrow}K\longrightarrow 1$
\end{center}
is a gyro-split extension, where $Aut (H)\times_{(\nu , \psi)}K$ is pull-back of the pair $(\nu , \psi )$ and $\nu:Aut(H)\rightarrow Out(H)$ is the natural group homomorphism.
\end{cor}

\begin{cor}\label{s5p3c2}
If $H$ is a group such that 
\begin{center}
$1\longrightarrow Inn (H)\stackrel{i}{\longrightarrow}Aut (H)\stackrel{\nu}{\longrightarrow}Out (H)\longrightarrow 1$
\end{center}
is a gyro-split exact sequence, then every extension of $H$ is a gyro-split extension. If in addition to this, $H$ has trivial center, then there is a unique (up to equivalence) such extension.
\end{cor}

\n For all finite simple groups $H$, the above sequence splits except when $H\ =\ A_{6}$. For $H\ =\ A_{6}$, the above sequence is not even a gyro-split extension. 

\vspace{0.2 cm}

\n A group is an internal semidirect product of its two subgroups if and only if the corresponding extension splits, that is the splitting is a group homomorphism. We now observe that the same is true in the case of gyro-splitting.

\begin{df}
Let $G$ be a group. We shall say that $G$ is internal gyro-semi direct product of a normal subgroup $H$ with a sub right loop $S$ of $(G, \circ_1)$ if $S$ is a right transversal to $H$ in $G$.
\end{df}

\n Thus, the exponent 3 non-abelian group $G$ of order $3^3$  is a gyro-semi direct product of its center with a sub loop of order $3^2$ of $(G, \circ_1)$. Evidently, a semidirect product is also a gyro-semi direct product. However, a gyro-semi direct product need not be a semidirect product.

\begin{thm}\label{s5p4}
A group $G$ is internal gyro-semi direct product of a normal subgroup $H$ with  a sub right loop $S$ of $(G, \circ_1)$ if and only if

(i) $G\ =\ HS$, and

(ii) $Hy^{2}\bigcap S\ =\ \{y^{2}\}$ (equivalently, $H\bigcap Sy^{2}\ =\ \{1\}$) for all $y\in S$.
\end{thm}

\begin{proof} Suppose that $G$ is internal gyro-semi direct product of a normal subgroup $H$ with  a sub right loop $S$ of $(G, \circ_1)$. Since $S$ is a right transversal, $G\ =\ HS$.   Given $y\in S$, since $S$ is a sub right loop of $(G, \circ_1)$, $y^{2}\in S$ and since $S$ is a right transversal,  $Hy^{2}\bigcap S\ =\ \{y^{2}\}$. 

Conversely, let $H$ be a normal subgroup of $G$, and $S$ be a sub right loop of $(G, \circ_1)$ such that the conditions $(i)$ and $(ii)$ hold. We need to show that $S$ is a right transversal. Already, $G\ =\ HS$. Suppose that $y^{-1}x\in H$, $x, y\in S$. Then $y^{-1}xy^{2}\in Hy^{2}\bigcap S\ =\ \{y^{2}\}$. This means that $y^{-1}x\ =\ 1$ and so $S$ is a right transversal to $H$ in $G$.
\end{proof}

\begin{rmk}
Unlike semidirect product, if $G$ is an internal gyro-semi direct product of $H$ with $S$ and it is also a gyro-semi direct product of $H$ with $T$, then $S$ need not be conjugate to $T$.
\end{rmk}

\n The following proposition is immediate.

\begin{prp}\label{s5p5}
$G$ is internal gyro-semi direct product of $H$ with a sub right loop of $(G, \circ_1)$ if and only if the exact sequence
\begin{center}
$1\longrightarrow H\stackrel{i}{\longrightarrow}G\stackrel{\nu}{\longrightarrow}G/H\longrightarrow 1$
\end{center}
is gyro-split.
\end{prp}

\n Next, let $H$ be an abelian group and $K\stackrel{\sigma}{\rightarrow}Aut(H)$ be an abstract kernel. Let $GEXT_{\sigma}(K, H)$ denote the set of equivalence classes of gyro-split extensions of $H$ by $K$ with abstract kernel $\sigma$. Obviously, $GEXT_{\sigma}(H, K)$ is non-empty, as the split extension exists which is also a gyro-split extension. Let $GZ^{2}_{\sigma}(K, H)$ denote the set of gyro-factor systems associated to $\sigma$. Evidently, $GZ^{2}_{\sigma}(K, H)$ is a subgroup of $Z^{2}_{\sigma}(K, H)$. We shall term $GZ^{2}_{\sigma}(K, H)$ as the group of gyro-cycles. Denote $B^{2}_{\sigma}(K, H)\bigcap GZ^{2}_{\sigma}(K, H)$ by $GB^{2}_{\sigma}(K, H)$ and call it the group of gyro-co-boundaries. We shall also term $GH^{2}_{\sigma}(K, H)\ =\ GZ^{2}_{\sigma}(K, H)/GB^{2}_{\sigma}(K, H)$ the second gyro-co-homology of $K$ with coefficients in $H$. From the proof of \cite[Proposition 10.1.11, p. 373]{lalb}, one can observe that given $(K, H, \sigma , f)\in  GZ^{2}_{\sigma}(K, H)$ there is the corresponding gyro-split extension of $H$ by $K$. The following proposition is easy to establish.

\begin{prp}\label{s5p6}
The map $\eta$ which associates to $(K, H, \sigma , f)\in  GZ^{2}_{\sigma}(K, H)$ the corresponding gyro-split extension induces a bijective map from  $GH^{2}_{\sigma}(K, H)$ to $GEXT_{\sigma}(K, H)$ which in turn, induces a group structure on $GEXT_{\sigma }(K, H)$.
\end{prp}
 
\n Further, it can be easily seen that the Baer sum in $EXT_{\sigma} (K, H)$ induces a sum in $GEXT_{\sigma}(K, H)$ with respect to which it is a subgroup isomorphic to $GH^{2}_{\sigma}(K, H)$. 
\begin{exm}
$GH^{2}_{\sigma}(\Z_{3}\times \Z_{3}, \Z_{3})\approx \Z_{2}$, whereas $H^{2}_{\sigma}(\Z_{3}\times \Z_{3} , \Z_{3})\approx V_{4}$. Here $\sigma$ is trivial.
\end{exm}

\n Given groups $H$ and $K$, $GHom(K, H)$ will denote the set of all gyro-homomorphisms from $K$ to $H$. If $H$ is an abelian group, then $GHom (K, H)$ is also an abelian group. Further, if $\alpha$ is a group homomorphism (gyro-homomorphism) from  a group $G$ to a group $K$ and $A$ is an abelian group, then $\alpha^{\star }$ is a  homomorphism from $GHom (K, A )$ to $GHom (G, A)$. Clearly, $GHom (K, A)$ is naturally isomorphic to $Hom (\check{K}, A)$. Consequently, we have the following proposition.

\begin{prp}\label{s5p7}
Let 
\begin{center}
$1\longrightarrow H\stackrel{\alpha}{\longrightarrow}G\stackrel{\beta}{\longrightarrow}K\longrightarrow 1$
\end{center}
be an exact sequence of groups. Let $A$ be an abelian group. Then the sequence
\begin{center}
$1\longrightarrow GHom (K, A)\stackrel{\beta^{\star}}{\rightarrow}GHom (G, A)\stackrel{\alpha^{\star}}{\rightarrow}GHom (H, A)$
\end{center}
is exact.
\end{prp}

%%%%%%%%%%%%%%%%%%%%%%%%%%%%%%%%%%%%%%%%%%%%%%%%%%%%%%%%%%%%%%%%%%%%%%%%%%%%%%%%%%%%%%%%%%%%

\section{Gyro-split central extensions and Gyro-Schur Multiplier}

Let $\cat{GRXT(-, K)}$ denote the category of gyro-split extensions by $K$. More explicitly, the objects of $\cat{GEXT(-, K)}$ are gyro-split short exact sequences 
\begin{center}
$E\equiv 1\longrightarrow H\stackrel{\alpha}{\longrightarrow}G\stackrel{\beta}{\longrightarrow}K\longrightarrow 1$
\end{center}
and a morphism from $E$ to 
\begin{center}
$E'\equiv 1\longrightarrow H'\stackrel{\alpha}{\longrightarrow}G'\stackrel{\beta}{\longrightarrow}K\longrightarrow 1$
\end{center}
is a pair $(\lambda , \mu )$ such that the triple $(\lambda , \mu , I_{K})$ is a morphism from $E$ to $E'$ in $\cat{GEXT}$. 

\n Let 
\begin{center}
$E\equiv 1\longrightarrow H\stackrel{\alpha}{\longrightarrow}G\stackrel{\beta}{\longrightarrow}K\longrightarrow 1$
\end{center}
be a gyro-split extension by $K$. Let $s$ be a gyro-splitting of $E$. Then $s$ is a gyro-homomorphism from $K$ to $G$. From the universal property of the pair $(\check{K},t_K)$, there is a unique group homomorphism $\mu$ from $\check{K}$ to $G$ such that $\mu\circ t_K=s$. In turn, $\beta \circ \mu \circ t_K=\beta\circ s=I_K=\nu_K\circ t_K$, where $\nu_K:\check{K} \rightarrow K$ is the natural homomorphism. Since $t_K(K)$ generates $\check{K}$, $\beta \circ \mu=\nu_K$. Thus, we get a group homomorphism $\lambda$ from $\check{R}_{K}$ to $H$ such that $(\lambda , \mu , I_{K})$ is a morphism from $E_{K}$ to $E$, where

\begin{center}
$E_K\equiv 1\longrightarrow \check{R}_{K} \stackrel{i_K}{\longrightarrow} \check{K}\stackrel{\nu_K}{\longrightarrow} K\longrightarrow 1$
\end{center}

\n More generally,  $E_{K}$ is a free gyro-split extension in the sense that given any gyro-split extension 
\begin{center}
$E'\equiv 1\longrightarrow H'\stackrel{\alpha}{\longrightarrow}G'\stackrel{\beta}{\longrightarrow}K'\longrightarrow 1$
\end{center}
and a gyro-homomorphism $\nu$ from $K$ to $K'$, there is a  pair $(\lambda , \mu )$ (not necessarily unique) such that $(\lambda , \mu , \nu )$ is a morphism from $E_{K}$ to $E'$. 

%%Moreover, it is also clear that
%%\begin{center}
%%$\check{E}_{K}\equiv 1\longrightarrow \check{R}_{K}/[\check{R}_{K}, \check{K}]\stackrel{\overline{i}}{\rightarrow}\check{K}/[\check{R}_{K}, \check{K}]\stackrel{\overline{\nu}}{\rightarrow}K\longrightarrow 1$
%%\end{center}
%%is a free gyro-split central extension in the sense that given any gyro-split central extension 
%%\begin{center}
%%$E'\equiv 1\longrightarrow H'\stackrel{\alpha '}{\rightarrow}G'\stackrel{\beta '}{\rightarrow}K'\longrightarrow 1$
%%\end{center}
%%and a gyro-homomorphism $\nu$ from $K$ to $K'$, there is a  $(\lambda , \mu )$ (not necessarily unique) of homomorphisms such that $(\lambda , \mu , \nu )$ is a morphism from $\check{E}_{K}$ to $E'$.

\vspace{0.2 cm}

\n The abstract kernel $\sigma$ associated to a central extension  is trivial. In this case, we shall denote $Z^{2}_{\sigma}(K, H)$ by $Z^{2}(K, H)$, $B^{2}_{\sigma}(K, H)$ by $B^{2}(K, H)$, $H^{2}_{\sigma}(K, H)$ by $H^{2}(K, H)$ and $GH^{2}_{\sigma}(K, H)$ by $GH^{2}(K, H)$. Let $A$ be an abelian group. We define a connecting group homomorphism $\delta$ from $ Hom(H, A)$ to $GH^{2} (K, A)$ as follows: Let $t$ be a gyro-splitting of $E$ and $f^{t}$ the corresponding gyro pairing in $GZ^{2}(K, H)$. Let $\eta\in Hom(H, A)$. Then $\eta \circ f^{t}$ is a map from $K\times K$ to $A$. Since $\eta$ is a group homomorphism, $\eta \circ f^{t}\in GZ^{2}(K, A)$. If $s$ is another gyro-splitting of $E$, then $f^{t}$ and $f^{s}$ differ by a member of $GB^{2}(K, H)$ and in turn, $\eta \circ f^{t}$ and $\eta \circ f^{s}$ differ by a member of $GB^{2}(K, A)$. This defines a group homomorphism $\delta$ from $Hom (H, A)$ to $GH^{2}(K, A)$ which is given by $\delta (\eta )\ =\ \eta \circ f^{t} + GB^{2}(K, A)$. 

\begin{prp}\label{s6p1}
For any abelian group $A$, we have the following natural fundamental exact sequence

\begin{center}
$0\longrightarrow Hom (K, A)\stackrel{\beta^{\star}}{\rightarrow}Hom (G, A)\stackrel{\alpha^{\star}}{\rightarrow}Hom (H, A)\stackrel{\delta}{\rightarrow}GH^{2}(K, A)$
\end{center}
associated to a gyro-split central  extension
\begin{center}
$E\equiv 1\longrightarrow H\stackrel{\alpha}{\longrightarrow}G\stackrel{\beta}{\longrightarrow}K\longrightarrow 1$.
\end{center}
\end{prp}
\begin{proof} Since $Hom$ is a left exact functor,  it is sufficient to prove the exactness at $Hom (H, A)$. Let $\chi\in Hom (G, A)$. By the definition, $\delta (\alpha^{\star}(\chi ))\ =\ (\chi \circ \alpha \circ f^{t})\ +\ GB^{2}(K, A)$. Already, $t(x)t(y)\ =\ \alpha (f^{t}(x, y))t(xy)$ for all $x, y\in K$ and since $t$ is a gyro-splitting, $f^{t}(y^{-1}, x) + f^{t}(y^{-1}x, y^{2})\ =\ 0$ for all $x, y\in K$. Since $\chi$ is a group homomorphism, $\chi (t(x))\ +\ \chi (t(y))\ =\ \chi (\alpha (f^{t}(x, y)))\ +\ \chi (t(xy))$. Thus, we have a map $g\ =\ \chi \circ t$ from $K$ to $A$ with $g(1)\ =\ 0$ and $(\chi \circ \alpha	)\circ f^{t}\ =\ \partial g$, where $\partial g(x,y)=g(y)-g(x,y)+g(x)$. This means that $\delta \circ \alpha^{\star}\ =\ 0$. It follows that $im( \alpha^{\star})\subseteq Ker \delta$. Next, let $\eta\in Ker \delta$. Then $\eta \circ f^{t}\in GB^{2} (K, A)$. Hence there is a map $g$ from $K$ to $A$ with $g(1)\ =\ 0$ such that
\begin{center}
$\eta (f^{t}(x, y))\ =\ g(y) - g(xy) + g(x)$
\end{center}
for all $x, y\in K$. Every element of $G$ is uniquely expressible as $\alpha (a)t(x),\ a\in H, x\in K$. Define a map $\chi$ from $G$ to $A$ by $\chi (\alpha (a)t(x))\ =\ \eta (a) + g(x)$. It can be easily seen that $\chi\in Hom (G, A)$ such that $\eta\ =\ \chi \circ \alpha\ =\ \alpha^{\star} (\chi )$. It follows that $Ker \delta\subseteq im( \alpha^{\star})$. 
\end{proof}

In particular, for an abelian group $H$, we have the following exact sequence:
\begin{center}
$0\longrightarrow Hom (K, H)\stackrel{\beta^{\star}}{\rightarrow}Hom (G, H)\stackrel{\alpha^{\star}}{\rightarrow}Hom (H, H)\stackrel{\delta}{\rightarrow}GH^{2}(K, H)$.
\end{center}

\begin{rmk}
The sequence
\begin{center}
$0\longrightarrow GHom (K, A)\stackrel{\beta^{\star}}{\rightarrow}GHom (G, A)\stackrel{\alpha^{\star}}{\rightarrow}GHom (H, A)\stackrel{\delta}{\rightarrow}GH^{2}(K, A)$.
\end{center}
need not be exact. Indeed, $\delta \circ \alpha^{\star}$ need not be $0$. However, $Ker\delta \subseteq im( \alpha^{\star})$.
\end{rmk}

\begin{prp}\label{s6p2}
The extension 
\begin{center}
$\check{E}_{K}\equiv  1\longrightarrow \check{R}_{K}/[\check{R}_{K}, \check{K}]\stackrel{\overline{i_{K}}}{\rightarrow}\check{K}/[\check{R}_{K}, \check{K}]\stackrel{\overline{\nu_{K}}}{\rightarrow}K\longrightarrow 1$
\end{center}
is a free gyro-split central extension of $K$ in the sense that given any gyro-split central extension
\begin{center}
$E'\equiv 1\longrightarrow H'\stackrel{\alpha '}{\longrightarrow}G'\stackrel{\beta '}{\longrightarrow}K'\longrightarrow 1$
\end{center}
and a gyro-homomorphism $\gamma$ from $K$ to $K'$, there is a pair $(\rho , \eta )$ (not necessarily unique) of homomorphism such that $(\rho , \eta , \gamma )$ is a morphism from $\check{E}_{K}$ to $E'$.
\end{prp}

\begin{proof} Evidently, $\check{E}_{K}$ is a gyro-split central extension. Again since $E_K$ is a free gyro-split  extension, there is a morphism $(\lambda , \mu , \gamma )$ from $E_{K}$ to $E'$. Since $E'$ is a central extension, $(\lambda , \mu )$ induces a pair $(\rho , \eta )$ such that $(\rho , \eta , \gamma )$ is a morphism from $\check{E}_{K}$ to $E'$.
\end{proof}

\begin{prp}\label{s6p3}
Let 
\begin{center}
$E\equiv 1\longrightarrow H\stackrel{\alpha}{\longrightarrow}G\stackrel{\beta}{\longrightarrow}K\longrightarrow 1$
\end{center}
be a free gyro-split central extension and $A$ be an abelian group. Then the map $\delta$ from $Hom (H, A)$ to $GH^{2}(K, A)$ is surjective. More explicitly, 
\begin{center}
$0\longrightarrow Hom (K, A)\stackrel{\beta^{\star}}{\rightarrow}Hom (G, A)\stackrel{\alpha^{\star}}{\rightarrow}Hom (H, A)\stackrel{\delta}{\rightarrow}GH^{2}(K, A)\longrightarrow 0$
\end{center} 
is exact.
\end{prp}

\begin{proof} Let $f\in GZ^{2}(K, A)$. Then $(K, A, \sigma , f)$ is a gyro-factor system with $\sigma$ being trivial. The corresponding associated extension 
\begin{center}
$E'\equiv 0\longrightarrow A\stackrel{\alpha '}{\longrightarrow}G'\stackrel{\beta '}{\longrightarrow}K\longrightarrow 1$
\end{center}
is a gyro-split central extension with a gyro-splitting $t'$ such that $t'(x)t'(y)\ =\ \alpha '(f(x, y))t'(xy)$ for all $x, y\in K$. Since $E$ is a free gyro-split central extension, we have a group homomorphism $\lambda$ from $H$ to $A$ and a group homomorphism $\mu$ from $G$ to $G'$ such that $(\lambda , \mu , I_{K})$ is a morphism from $E$ to $E'$. Let $t$ be a  gyro-splitting of $E$. Then $\beta '(\mu (t(x)))\ =\ \beta (t(x))\ =\ x$ for all $x\in K$. Hence $t''\ =\ \mu \circ t$ is a gyro-splitting of $E'$. Thus, $f^{t''} + GB^{2}(K, A)\ =\ f + GB^{2}(K, A)$. Now, $t(x)t(y)\ =\ \alpha (f^{t}(x, y))t(xy)$ for all $x,y\in K$. Further, 
\begin{center}
$\alpha '(f^{t''}(x, y))t''(xy)\ =\ t''(x)t''(y)\ =\ \mu (t(x))\mu (t(y))\ =\ \mu (t(x)t(y))\ =\ \mu (\alpha (f^{t}(x, y)))\mu (t(xy))\ =\ \mu (\alpha (f^{t}(x, y)))t''(xy)\ =\ \alpha '(\lambda (f^{t}(x, y)))t''(xy)$.
\end{center}
This shows that $\alpha '(\lambda (f^{t}(x, y)))\ =\ \alpha '(f^{t''}(x, y))$. Since $\alpha '$ is injective, $\lambda (f^{t}(x, y))\ =\ f^{t''}(x, y)$. By the definition $\delta (\lambda )\ =\ f^{t''} + GB^{2}(K, A)\ =\ f + GB^{2}(K, A)$. This shows that $\delta$ is surjective.
\end{proof}

\begin{prp}\label{s6p4}
Let 
\begin{center}
$E\equiv 1\longrightarrow H\stackrel{\alpha }{\longrightarrow}G\stackrel{\beta }{\longrightarrow}K\longrightarrow 1$
\end{center}
be a gyro-split central extension by $K$, and $D$ be a divisible abelian group. Then the image of $\delta$ in the fundamental exact sequence
\begin{center}
$0\longrightarrow Hom (K, D)\stackrel{\beta^{\star}}{\rightarrow}Hom (G, D)\stackrel{\alpha^{\star}}{\rightarrow}Hom (H, D)\stackrel{\delta}{\rightarrow}GH^{2}(K, D)$
\end{center}
is isomorphic to $Hom ([G, G]\bigcap \alpha (H), D)$. In particular, if the extension $E$ is a free gyro-split central extension, then $GH^{2}(K, D)$ is isomorphic to $Hom ([G, G]\bigcap \alpha (H), D)$.
\end{prp}

\begin{proof} By the fundamental theorem of homomorphism, 
\[im(\delta)\ \approx Hom (H, D)/Ker \delta\ = Hom (H, D)/im( \alpha^{\star}).\] 

The map $\alpha$ induces an injective group homomorphism $\overline{\alpha}$ from $H/(H\bigcap \alpha^{-1}([G, G])$ to $G/[G, G]$. Since $D$ is divisible, $\overline{\alpha}^{\star}$ is a surjective group homomorphism from $Hom (G/[G, G], D)$ to $Hom ( H/(H\bigcap \alpha^{-1}([G, G]), D))$. Also, since $D$ is abelian, $\nu^{\star}$ from $Hom (G/[G, G], D)$ to $Hom (G, D)$ is an isomorphism, where $\nu:G\rightarrow G/[G,G]$ is the quotient map.  Further, $\rho^{\star}o\overline{\alpha^{\star}}\ =\ \alpha^{\star}o\nu^{\star}$, where $\rho$ is the quotient map from $H$ to $H/(H\bigcap \alpha^{-1}([G, G]))$. It follows that the image of $\alpha^{\star}$ is that of $\rho^{\star}$. Again, since $D$ is divisible, the following sequence is exact:
\begin{center}
$0\longrightarrow Hom ( H/(H\bigcap \alpha^{-1}([G, G]), D))\stackrel{\rho^{\star}}{\rightarrow}Hom (H, D)\stackrel{i^{\star}}{\rightarrow}Hom ( (H\bigcap \alpha^{-1}([G, G])), D)\longrightarrow 0$.
\end{center}
Thus, 
\begin{center}
$Hom (H, D)/im(\rho^{\star})\ \approx Hom ( (H\bigcap \alpha^{-1}([G, G])), D)\ \approx\ Hom (([G, G]\bigcap \alpha (H)), D)$.
\end{center}
The last assertion follows from the proposition \ref{s6p3}.
\end{proof}
 
\begin{cor}\label{s6p4c1}
%\begin{center}
$GH^{2}(K, \C^{\star})\approx Hom (([\check{K}, \check{K}]\bigcap \check{R}_{K} )/[\check{K}, \check{R}_{K}], \C^{\star})$.
%\end{center}
More generally, if $\langle X; S \rangle$ is a presentation of $K$, then
\begin{center}
$GH^{2}(K, \C^{\star})\approx Hom (([F(F(X)), F(F(X))]\bigcap \hat{S} )/[F(F(X)), \hat{S}], \C^{\star})$.
\end{center}
\end{cor}

\n Since $GH^{2}(K, \C^{\star})$ is a subgroup of $H^{2}(K, \C^{\star})$, the following corollary is a consequence of the Schur-Hopf Formula.

\begin{cor}\label{s6p4c2}
If $K$ is finite, then
\begin{center}
$GH^{2}(K, \C^{\star})\approx  [F(F(X)), F(F(X))]\bigcap \hat{S} /[F(F(X)), \hat{S}]$.
\end{center}
\end{cor}

\n We shall term $GH^{2}(K, \C^{\star})$ and also $([\check{K}, \check{K}]\bigcap \check{R}_{K} )/[\check{K}, \check{R}_{K}]$  as gyro-Schur Multipliers of $K$. Note that they are same provided that $K$ is finite. Also observe that $K\mapsto ([\check{K}, \check{K}]\bigcap \check{R}_{K} )/[\check{K}, \check{R}_{K}]$ defines a functor from $\cat{GP}$ to itself. 

\vspace{0.2 cm}

\n The proof of the following proposition is an easy verification.

\begin{prp}
Let $K$ be a group. Then the right gyro-group operation $\circ_1$ on $K$ satisfies the following relations:

\begin{enumerate}
	\item[$(i)$] $(xy)\circ_1 z\ =\ x^{z}(y\circ_1 z)$, and also
	\item[$(ii)$] $x\circ_1 (yz)\ =\ (x^{y}\circ_1 z)y^{z}$.
\end{enumerate}

\n for each $x,y, z\in K$, where $x^y=y^{-1}xy$.
\end{prp}

\n The relations described in the above propositions will be termed as trivial relations for $\circ_1$. Recall that the Schur multiplier of a group $K$ has description as the group of non-trivial  commutator relations of $K$ \cite{lalb,ml}. We describe the gyro-Schur multiplier $([\check{K}, \check{K}]\bigcap \check{R}_{K} )/[\check{K}, \check{R}_{K}]$ also as the group of non-trivial relations of the right gyro-group operation $\circ_1$ of $G$.

\vspace{0.2 cm} 

\n Let $K$ be a group. Let $K\boxtimes K$ denote the abelian group generated by the set $\{x\boxtimes y\mid x, y\in K\}$ subject to the relations 
\begin{enumerate}
	\item[(i)] $1\boxtimes x\ =\ 1\ =\ x\boxtimes 1$,
	\item[(ii)] $(x\boxtimes y)((xy)\boxtimes z)\ =\ (y\boxtimes z)((x\boxtimes (y z)))$ and
	\item[(iii)] $(y^{-1}\boxtimes x)((y^{-1}x)\boxtimes y^{2})=1$,
\end{enumerate}
\n for all $x,y,z\in K$. We shall term $K\boxtimes K$ as gyro-square of $K$.

\begin{thm}\label{s6t1}
We have a free gyro-split central extension 
\begin{center}
$U\equiv\  1\longrightarrow K\boxtimes K\stackrel{i_{1}}{\longrightarrow} (K\boxtimes K)\rtimes K\stackrel{p_{2}}{\longrightarrow}K\longrightarrow 1$,
\end{center}
where $(K\boxtimes K)\rtimes K$ is a group with respect to the operation given by $(a, x)(b, y)\ =\ (ab(x\boxtimes y), xy)$.   
\end{thm}

\begin{proof}  Let 
\begin{center}
$E'\equiv\  1\longrightarrow H'\stackrel{\alpha '}{\longrightarrow}G'\stackrel{\beta '}{\longrightarrow}K'\longrightarrow 1$
\end{center}
be a gyro-split central extension, and $\nu$ be a gyro-homomorphism from $K$ to $K'$.  Let $t$ be a gyro-splitting of $E'$, and $(K', H', \sigma^{t}, f^{t} )$ be the corresponding factor system. Then $\sigma^{t}$ is trivial. Further, $f^{t}(x, y)f^{t}(xy, z)\ =\ f^{t}(y, z)f^{t}(x, yz)$ and since $t(y^{-1}xy^{2})\ =\ t(y)^{-1}t(x)t(y)^{2}$ for all $x, y\in K'$, $f^t(y^{-1}, x)f^t(y^{-1}x, y^{2})\ =\ 1$. Thus, we have a group homomorphism $\lambda$ from $K\boxtimes K$ to $H'$ given by $\lambda (x\boxtimes y)\ =\ f^t(x, y) $. In turn, we have a map $\mu$ from $(K\boxtimes K)\rtimes K$ to $G'$ given by $\mu (a, x)\ =\ \alpha'(\lambda (a))t(\nu (x))$. It can be seen that $\mu$ is a group homomorphism and $(\lambda ,\mu , \nu )$ is a morphism. 
\end{proof}

\begin{cor}\label{s6t1c1}
The extension $\check{E}_{K}$ as described in the Proposition \ref{s6p2} is equivalent to $U$. 
\end{cor}

\begin{proof} Since the map $x\mapsto (1, x)$ is a gyro-homomorphism from $K$ to $(K\boxtimes K)\rtimes K$, it induces a group homomorphism $\mu$ from $\check{K}$ to $(K\boxtimes K)\rtimes K$ given by $\mu (x \check{S}_{K})\ =\ (1,x)$. It can be easily observed that $[\check{R}_{K}, \check{K}]$ is contained in the kernel of $\mu$. This in turn induces a morphism from $\check{E}_{K}$ to $U$. Further, Theorem \ref{s6t1} gives the inverse of this morphism. 
\end{proof}  

\n Let $K\stackrel{\sigma}{\rightarrow}Aut (H)$ be an abstract kernel, where $H$ is an abelian group. Let 

\begin{center}
$1\longrightarrow H\stackrel{i}{\longrightarrow}G\stackrel{\nu}{\longrightarrow}K\longrightarrow 1$
\end{center}
be a gyro-split extension of $H$ by $K$ which is associated to $\sigma$. Note that it is central extension if and only if $\sigma$ is trivial. We denote the image $\sigma(x)$ by $\sigma_x$. Consider the subset $A=\{h\in H\mid \sigma_x(h)=h, ~ \forall x\in K\}$. Evidently, $A$ is a central subgroup of $G$ and we have the following commutative diagram.

\[
\begin{tikzcd}
  1 \arrow[r] & A \arrow[d, "i"] \arrow[r, "i"] & G \arrow[d, "I_G"] \arrow[r, "\tilde{\nu}"] & G/A \arrow[d, "\hat{\nu}"] \arrow[r] & 1 \\
  1 \arrow[r] & H  \arrow[r, "i"] & G \arrow[r, "\nu"] & K  \arrow[r] & 1 
\end{tikzcd}
\]

\n where the top row is a gyro-split central extension of $A$ by $G/A$ and the maps are the obvious maps. Indeed, if $t$ is a gyro-splitting of the bottom row, then $t\circ \hat{\nu}$ is a gyro-splitting of the top row. From the proof of the Theorem \ref{s6t1}, we have a morphism from the extension $U$ to the extension given in the top row and in turn, we have a morphism $(\chi , \psi , I_K )$ from $U$ to the given gyro-split extension

\begin{center}
$1\longrightarrow H\stackrel{i}{\longrightarrow}G\stackrel{\nu}{\longrightarrow}K\longrightarrow 1$
\end{center}

\n with $\chi (K\boxtimes K)\subseteq A$. Conversely, let $\chi $ be a group homomorphism from $K\boxtimes K$ to $A\subseteq H$. Then $(K, H, \sigma , \tilde{\chi} )$ is a factor system, where $\tilde{\chi}$ is a map from $K\times K$ to $H$ given $\tilde{\chi}(x,y)=\chi(x\boxtimes y)$. The corresponding extension 
\begin{center}
$E_{\chi}\equiv\ 1\longrightarrow H\stackrel{i_{1}}{\longrightarrow}L\ =\ H\times K\stackrel{p_{2}}{\longrightarrow}K\longrightarrow 1$ 
\end{center}
is a gyro-split extension of $H$ by $K$ with $x\mapsto (1, x)$ as a gyro-splitting. Thus, we have a surjective map $\lambda$ from $Hom (K\boxtimes K , A)$ to $GEXT_{\sigma}(H, K)$ given by $\lambda (\chi )\ =\ [E_{\chi}]$. Clearly, $\lambda$ is also a group homomorphism. We describe the $Ker \lambda$. Now, $\chi\in Ker \lambda$ if and only if the corresponding factor system is equivalent to the trivial factor system. In other words, there is a map $g$ from $K$ to $H$ with $g(1)\ =\ 0$ such that 
$\chi (x\boxtimes y)\ =\ \partial g(x, y)\ =\ \sigma_{x}(g(y))\ -\ g(xy)\ +\ g(x) $ belongs to $A$ for all $x, y\in K$. Evidently, $(K, H, \sigma , \partial g)$ is a gyro-factor system. Let us call such a map $g$ to be a gyro-crossed homomorphism relative to $\sigma$. Thus an identity preserving map $g$ from $K$ to $H$ is a gyro-crossed homomorphism if
\begin{center}
$\sigma_{x}(\sigma_{y}(g(z))\ -\ g(yz)\ +\ g(y))\ =\ \sigma_{y} (g(z))\ -\ g(yz)\ +\ g(y)$,
\end{center}
and
\begin{center}
$\sigma_{y^{-1}}(g(x))\ +\ g(y^{-1})\ +\ \sigma_{y^{-1}x}(g(y^{2}))\ -\ g(y^{-1}xy^{2})\ =\ 0$
\end{center}
for all $x, y, z\in K$. Evidently, every crossed group homomorphism is a gyro-crossed homomorphism. However, a gyro-crossed homomorphism need not be a crossed group homomorphism. For example, if $K$ is the exponent 3 non-abelian group of order $3^3$, then the map $g$ from $K$ to $K\boxtimes K$ given by $g(x)\ =\ x\boxtimes x$ can be easily seen to be a gyro-crossed homomorphism which is not a crossed group homomorphism. Let $GC_{\sigma}(K, H)$ denote the group of all gyro-crossed homomorphisms from $K$ to $H$. The above discussion  establishes the following proposition.

\begin{prp}\label{s6p5}
A map $g$ with $g(1)\ =\ 0$ is a gyro-crossed homomorphism from $K$ to $H$ relative to $\sigma$ if and only if $(K, H, \sigma , \partial g)$ is a gyro-factor system
and $\partial g(K\times K)\subseteq A$. In turn, $\partial g$ induces a homomorphism $\overline{\partial }$ from $GC_{\sigma}(K, H)$ to $Hom (K\boxtimes K , A)\subseteq Hom (K\boxtimes K , H)$ given by $\overline{\partial g} (x\boxtimes y)\ =\ \partial g(x, y)$, and we have the exact sequence
\begin{center}
$0\rightarrow C_{\sigma}(K, H)\stackrel{i}{\rightarrow}GC_{\sigma}(K, H)\stackrel{\overline{\partial }}{\rightarrow}Hom (K\boxtimes K, A)\stackrel{\lambda}{\rightarrow}GEXT_{\sigma} (K, H)\rightarrow 0$,
\end{center}
where $C_{\sigma}(K, H)$ denotes the group of crossed homomorphisms. 
\end{prp}

In case $\sigma$ is trivial or equivalently, it is a central extension, then we omit $\sigma$ in the notation. In particular, we have the following  exact sequence:
\begin{center}
$0\rightarrow Hom (K, H)\stackrel{i}{\rightarrow}GC(K, H)\stackrel{\overline{\partial}}{\rightarrow}Hom (K\boxtimes K, H)\stackrel{\lambda}{\rightarrow}GEXT(K, H)\rightarrow 0$.
\end{center}

 %% \n Looking at the central extensions by $E$, we can compute the gyro-square $E\boxtimes E$, where $E$ is the non abelian group of order 27.

%%%%%%%%%%%%%%%%%%%%%%%%%%%%%%%%%%%%%%%%%%%%%%%%%%%%%%%%%%%%%%%%%%%%%%%%%%%%%%%%%%%%%%%%%%

\section{Universal free gyro-split central extension, Milnor gyro-$K_{2}$ group}
\begin{df}
A gyro-split central extension 
\begin{center}
$\Omega_{K}\equiv\ 1\longrightarrow H\stackrel{i}{\longrightarrow}U\stackrel{j}{\longrightarrow}K\longrightarrow 1$
\end{center}
will be termed as a universal free gyro-split central extension by $K$ if given any gyro-split central extension
\begin{center}
$E\equiv\ 1\longrightarrow L\stackrel{\alpha}{\longrightarrow}G\stackrel{\beta}{\longrightarrow}K\longrightarrow 1$
\end{center}
by $K$, there is a unique group homomorphism $\phi$ from $U$ to $G$ inducing a morphism $(\xi , \phi , I_{K})$ from $\Omega_{K}$ to $E$.
\end{df}

\n Evidently, a universal free gyro-split central extension by $K$ (if exists) is unique up to equivalence.

\begin{prp}\label{s7t1}
If 
\begin{center}
$\Omega_{K}\equiv\ 1\longrightarrow H\stackrel{i}{\longrightarrow}U\stackrel{j}{\longrightarrow}K\longrightarrow 1$
\end{center}
is a universal free gyro-split central extension by $K$, then $U$ is perfect. In particular, $K$ is perfect.
\end{prp}

\begin{proof} Suppose that $U$ is not perfect. Then $U/[U, U]$ is a non-trivial abelian group. Consider the direct product extension
\begin{center}
$1\longrightarrow U/[U, U]\stackrel{i_{1}}{\longrightarrow} U/[U, U]\times K\stackrel{p_{2}}{\longrightarrow}K\longrightarrow 1$.
\end{center}
Clearly, this extension is a gyro-split (indeed, a split) central extension. Further, the map $(\nu , j)$ from $U$ to $U/[U, U]\times K$ defined by $(\nu , j)(u)\ =\ (u[U, U], j(u))$ and $(0, j)$ given by $(0 , j)(u)\ =\ ([U, U], j(u))$ are two group homomorphisms inducing morphisms  from $\Omega_{K}$ to this extension.  This is a contradiction. This shows that $U$ is perfect. Consequently, $K$ is perfect. 
\end{proof}

\n Let us call a gyro-homomorphism $f$ from a group $G$ to a group $K$ to be a strong gyro-homomorphism if $f$ preserves the commutator operation in the sense that $f([a, b])\ =\ [f(a), f(b)]$ for all $a, b\in G$. An extension $E$ is said to be a strong gyro-split extension if it has a section $t$ which is a strong gyro-homomorphism. We have a category $\cat{SGP}$ whose objects are groups and morphism between groups are strong gyro-homomorphisms. Obviously, the category $\cat{GP}$ of groups is a faithful (but not full) subcategory of $\cat{SGP}$. We construct the adjoint to the inclusion functor from $\cat{GP}$ to $\cat{SGP}$.

\vspace{0.2 cm}

\n Let $K$ be a group. Consider the free group $F(K)$ on $K$ and standard group homomorphism $\rho$ from $F(K)$ to $K$ which is the identity map on $K$. Let $SG(K)$ denote the set $\check{S}_{K}\bigcup \{(xyx^{-1}y^{-1})^{-1}\star x\star y\star x^{-1}\star y^{-1}\mid x, y\in K\}$ of words in $F(K)$, and $\check{SG}(K) $ denote the group having the presentation $\langle K; SG(K) \rangle$. More explicitly, $\check{SG}(K)\ =\ F(K)/\langle SG(K) \rangle $. It follows from the construction that the association $K\mapsto \hat{SG(K)}$ defines a functor from $\cat{GP}$ to $\cat{SGP}$, which is adjoint to the forgetful functor from $\cat{SGP}$ to $\cat{GP}$. Clearly, $\langle R_{K} \rangle\supseteq \langle SG(K) \rangle$. Further, we have a strong gyro-split extension
\begin{center}
$\tilde{E}_{K}\equiv\ 1\longrightarrow \langle R_{K} \rangle / \langle SG(K) \rangle\stackrel{i}{\longrightarrow}\check{SG}(K)\stackrel{\nu}{\longrightarrow} K\longrightarrow 1$.
\end{center}

Evidently, $\tilde{E}_K$ is a free strong gyro-split extension by $K$. We may term $\langle R_{K} \rangle / \langle SG(K) \rangle$ as a strong gyro-multiplier. Note again that if $K$ is free on a set having at least two elements, the strong gyro-multiplier is non-trivial.

\begin{prp}\label{s7t2}
Let $K$ be a perfect group in which every element is a commutator. Then $K$ admits a universal free gyro-split central extension.
\end{prp}

\begin{proof} Let $K$ be a perfect group in which every element is a commutator. Consider the strong gyro-split extension
\begin{center}
$\tilde{E}_{K}\equiv\ 1\longrightarrow \check{R}_{K}\ =\  \langle R_{K} \rangle / \langle SG(K) \rangle \stackrel{i}{\longrightarrow}\check{SG}(K)\stackrel{\nu}{\longrightarrow} K\longrightarrow 1$.
\end{center}

\n having a strong gyro-splitting $t$ given by $t(x)\ =\ x\langle R_{K} \rangle / \langle SG(K) \rangle$. Since every element of $K$ is a commutator, image of $t$ is contained in $[\check{SG}(K),\check{SG}(K)]$. In turn, we get a gyro-split central extension
\begin{center}
$\check{\check{E}}_{K}\equiv\ 1\longrightarrow  (\check{R}_{K}\bigcap [\check{SG}(K), \check{SG}(K)])/[\check{R}_{K},\check{SG}(K) ] \stackrel{i}{\rightarrow}[\check{SG}(K), \check{SG}(K)] /[\check{R}_{K}, \check{SG}(K)]\stackrel{\nu}{\rightarrow}K\longrightarrow 1$.
\end{center}
We show that $\check{\check{E}}_{K}$ is universal free gyro-split central extension. Let
\begin{center}
$E\equiv\ 1\longrightarrow H\stackrel{i}{\longrightarrow}G\stackrel{\beta}{\longrightarrow}K\longrightarrow 1$
\end{center}
be a gyro-split central extension by $K$. Since $E_{K}$ is a free gyro-split extension by $K$, there is a homomorphism $\phi$ from $\check{K}$ to $G$ which induces a morphism $(\phi |_{\check{R}_{K}}, \phi , I_{K})$ from $E_{K}$ to $E$. Further, since $K$ is perfect, $\beta\mid_{[G,G]}$ is a surjective group homomorphism. In turn, we get a central extension 
\begin{center}
$E'\equiv\ 1\longrightarrow H\bigcap [G, G]\stackrel{i}{\longrightarrow}[G, G]\stackrel{\beta}{\longrightarrow}K\longrightarrow 1$.
\end{center}
It follows from the construction that $\phi$ induces a group homomorphism from $[\check{SG}(K),\check{SG}(K)] /[\check{R}_{K}, \check{SG}(K)]$ to $[G, G]$ which, in turn, induces a morphism from  $\check{\check{E}}_K$ to $E'$. Since $K$ is perfect, $[\check{SG}(K), \check{SG}(K)] /[\check{R}_{K}, \check{SG}(K)]$ is also perfect. Consequently, the induced morphism is unique (see \cite[Proposition 10.4.2]{lalb}).
\end{proof}

\begin{cor}\label{s7t2c1}
(i) Every finite simple group admits a universal free gyro-split central extension.

(ii) $SU(n)$ admits a universal free gyro-split central extension.
\end{cor}

\begin{proof} The proof of the Ore's conjecture \cite{ml} implies (i), while the fact that every element of $SU(n)$ is a commutator \cite{ht} implies (ii). 
\end{proof}

\begin{rmk}
It is not clear if every perfect group admits a universal free gyro-split central extension.
\end{rmk}

We have the following gyro analogues of non-abelian exterior square, Steinberg group, and Milnor $K_{2}$.
\begin{df}
We shall term $[\check{SG}(K), \check{SG}(K)] /[\check{R}_{K}, \check{SG}(K)]$ as a non-abelian gyro-exterior square of $K$ and denote it by $K\bigwedge^{G} K$. If $K$ is perfect, we have the universal free gyro-split central extension
\begin{center}
$1\longrightarrow M^{G}(K)\stackrel{i}{\longrightarrow}K\bigwedge^{G}K\stackrel{\nu}{\longrightarrow}K\longrightarrow 1$,
\end{center}
where $M^{G}(K)\ =\ (\check{R}_{K}\bigcap [\check{SG}(K), \check{SG}(K)])/[\check{R}_{K}, \check{SG}(K)]$ is gyro-Schur multiplier of $K$. Further, for any ring $R$ with identity, we have the invariant $St^{G}(R)\ =\ E(R)\bigwedge^{G} E(R)$  termed as gyro-Steinberg group over $R$ and the group $K^{G}_{2}(R)\ =\ M^{G}(E(R))$ termed as gyro-Milnor group. 
\end{df}

We have the exact sequence
\begin{center}
$1\longrightarrow K^{G}_{2}(R)\longrightarrow St^{G}(R)\longrightarrow E(R)\longrightarrow 1$.
\end{center}

\n \textbf{Acknowledgment:} Authors are extremely grateful to the reviewer for his/her fruitful comments.

\end{document}